\documentclass[draft]{amsart}

\usepackage{amssymb}
\usepackage{url}
\usepackage{amscd}
\usepackage{enumerate}
\usepackage[all]{xy}

\def \R{{\mathbb R}}

\def \1{{\mathbb 1}}

\theoremstyle{plain}
\newtheorem{theorem}{Theorem}
\newtheorem{proposition}{Proposition}
\newtheorem{definition}{Definition}
\newtheorem{lemma}{Lemma} 
\newtheorem{corollary}{Corollary}

\theoremstyle{remark}
\newtheorem{remark}{Remark}

\newtheorem{Exemp}{Example}

\begin{document}
\title[The Lagrange Multipliers]{Lagrange Multipliers in  locally convex spaces.}
\author{Mohammed Bachir, Joel Blot}
\date{\today}
\address{Laboratoire SAMM 4543, Universit\'e Paris 1 Panth\'eon-Sorbonne, France}

\email{Mohammed.Bachir@univ-paris1.fr}

\email{Joel.Blot@univ-paris1.fr}

\maketitle

\begin{abstract}
We give a general Lagrange multiplier rule for mathematical programming problems in a Hausdorff  locally convex space. We consider infinitely many inequality and equality constraints. Our results gives in particular a generalisation of  the result of J. Jahn in \cite{Ja}, replacing Fr\'echet-differentiability assumptions on the functions by the Gateaux-differentiability. Moreover, the closed convex cone with a nonempty interior  in the constraints is replaced by a strictly  general class of closed subsets introduced in the paper and called {\it ``admissible sets"}. Examples illustrating our results are given.
\end{abstract}

{\bf Keywords:} Lagrange multipliers, Optimization problems, Admissible sets, Equi-Gateaux-differentiability.

{\bf 2020 Mathematics Subject Classiﬁcation:} \subjclass{Primary 46N10, 49J50, Secondary 46G05} 

\section{Introduction.}
Let $E$ be an Hausdorff locally convex topological vector  space (in short l.c.t.v space, the term ``Hausdorff'' will be implicit) and $\Omega\subset E$ be an open subset. Let $f: \Omega \to \R$ be a function. The aim of this paper consists in giving  a necessary condition, using Lagrange multipliers, for a point $\hat{x}\in \Omega$ to be a solution of the following optimization problem with (finite or infinite) inequality constraints
\begin{equation*}
(\mathcal{P})
\left \{
\begin{array}
[c]{l}
\max f\\
x\in \Omega \\
x\in A:=[C]^\times:=\lbrace x\in E: \phi(x)\geq 0, \hspace{1mm} \forall \phi \in C\rbrace
\end{array}
\right. 
\end{equation*}
where $C$ is a  set of  functions $\phi: E \to \R$. We prove in our first main results (Theorem \ref{FarkasMM} and Proposition \ref{FarkasM})  that a natural condition to obtain  non-trivial Lagrange multipliers for the problem $(\mathcal{P})$ in a genaral  l.c.t.v space $E$  with finite or infinite  inequality constraints, is that:


$\bullet$ The function $f$ is Gateaux differentiable at the optimal solution $\hat{x}$ and the family $C$ is equi-Gateaux differentiable at $\hat{x}$ (i.e. Gateaux differentiable at $\hat{x}$ with a same modulus) and the set $\lbrace \phi \in C: \phi(\hat{x})\neq 0 \rbrace$ is equi-lower semicontinuous at $\hat{x}$ (see, the definitions in Subsection \ref{Equi-G}). We also assume  that $\inf_{\phi\in C}\phi(\hat{x})=0$, otherwize $\hat{x}$ belongs to the interior of $A$ and so there is non constraints.

$\bullet \bullet$ The weak-star closed convex hull  $\overline{\textnormal{conv}}^{w^*}\lbrace d_G \phi (\hat{x}): \phi \in C \rbrace$ in the topological dual $E^*$, is $w^*$-compact (where, $d_G \phi (\hat{x})$ denotes the Gateaux-differential of a function $\phi$ at $\hat{x}$). 


A set $A=[C]^\times$ where $C$ satisfies the above conditions, will be said {\it weak-admissible} at $\hat{x}$ (Definition \ref{def1}). 

Under these natural conditions, we obtain  non-trivial Lagrange multipliers as follows: there exists  $(\lambda^*, \beta^*)\in \R^+\times \R^+$ such that $(\lambda^*,\beta^*)\neq (0,0)$ and 
$
\lambda^* d_G f  (\hat{x}) \in -\beta^* \mathcal{T}_{C}(\hat{x}),
$
(``condition of Fritz John'') where, $\mathcal{T}_{C}(\hat{x}):=\cap_{n\geq 1} \overline{\textnormal{conv}}^{w^*}\lbrace d_G\phi(\hat{x}): \phi(\hat{x}) \in [0, \frac{1}{n}], \phi\in C\rbrace \subset E^*$ .
If moreover we assume that $0 \not \in \mathcal{T}_{C}(\hat{x})$, then we can assume that $\lambda^*=1$ (``condition of Karush-Kuhn-Tucker''). 

The set $\mathcal{T}_{C}(\hat{x})$ is small enough to encompass known results, such as when the set of constraints is finite or when $A=[C]^\times$ is a closed convex set (see Example ~\ref{ExN2} and Proposition ~\ref{nonzero}). In addition, the set $\mathcal{T}_{C}(\hat{x})$ defined above is optimal in the sens that it cannot be replaced in general (when $C$ is infinite) by the set $\overline{\textnormal{conv}}^{w^*}\lbrace d_G\phi(\hat{x}): \phi(\hat{x})=0, \phi\in C\rbrace$ (a simple example even in $\R^2$ is given in Example \ref{ExN1}). Examples of more explicit sets containing $\mathcal{T}_{C}(\hat{x})$ are given in Example \ref{ExN2} and Example \ref{ExN2bis}.


The above mentioned result extends  to  infinitely many inequality constraints  in l.c.t.v spaces under the Gateaux-differentiability assumption, results established for finitely many inequality constraints in finite dimension (see for instance the works of  P. Michel in \cite{Mp}, J. Blot in \cite{Bl} and H. Yilmaz in \cite{Yi}). Our work also generalises the so called semi-infinite programming (SIP) problems (see Example \ref{ExN2bis} and Corollary \ref{semi-finite}).
\vskip5mm
On the other hand, we prove in our second main results (Theorem \ref{Farkas-bis} and Corollary \ref{cor-Farkas-bis}) a Lagrange multiplier rule  in Banach spaces for mathematical programming with both infinitely many inequality and equality constraints of the form 
\begin{equation*}
(\widetilde{\mathcal{P}})
\left \{
\begin{array}
[c]{l}
\max f\\
x\in \Omega \\
g(x) \in A\\
h(x)=0
\end{array}
\right. 
\end{equation*}
where, $E$ and $W$ are Banach spaces, $\Omega$ an open subset of $E$, $Y$ is a normed  space, $A\subset Y$ and   $g: \Omega \to Y$, $h: \Omega \to W$ and $f: \Omega \to \R$  are  mappings.

There are several works in the literature addressing this subject with different conditions (see for instance \cite{Bl1, Ja, JS, JT1,JT2, Ku}. For the convex and affine frame, we refer to \cite{Dmb}). The classical conditions given by J. Jahn in \cite[Theorem 5.3 ]{Ja} assume that the functions $f$, $g$ and $h$ are Fr\'echet differentiable  and that the set $A$ is a closed convex cone with a nonempty interior.  In Theorem \ref{Farkas-bis} and Corollary \ref{cor-Farkas-bis}, we generalize  the result  given by J. Jahn \cite{Ja} in the following directions:

$\bullet$  The objective function $f$ and the function $g$  in the constraint are assumed to be  G\^ateaux  differentiable at the optimal solution not necessarilly Fr\'echet differentiable at this point. 

$\bullet \bullet$ We extend  the assumption that $A$ is a closed convex cone with nonempty interior in \cite[Theorem 5.3 ]{Ja} to a more general class of closed subsets which are not necessarily neither cone nor even convex but includes the class of all closed convex subset $A$  whose recession cone $\mathcal{R}_A$ (see Section \ref{prel}) has a nonempty interior, it includes in particular closed  convex cones with nonempty interiors (see Corollary \ref{cor-Farkas-bis}). This class of sets will be introduced in Section \ref{Adm}, which we will call the class of {\it "admissible sets"} (Definition \ref{def2}).

\vskip5mm

In infinite dimension, most of the authors assumed that $A$ is a closed convex cone with nonempty interior. The first result which gives conditions in the case of closed sets is due to Jourani and Thibault \cite{JT1,JT2, Jourani}, dealing with the approximate subdifferential (see \cite{Ioffe1, Ioffe2}).  Our conditions are on the one hand  different from those given in \cite{JT1,JT2, Jourani} and on the other hand cannot be deduced from the theory of approximate subdifferential because, just like the Clarke's subdifferential, the approximate subdifferential of a Gateaux differentiable function at some point does not generally coincide with the Gateaux-differential of the function at this point. 

\vskip5mm

This paper is organized as follows. In Section \ref{prel}, we recall some classical notions  around convex cones and equi-differentiability and we give some examples. In Section \ref{Adm}, we introduce the notion of {\it admissible sets}, we give non-convex examples of such sets and prove that a closed convex cones with a nonempty interior is a particular {\it admissible set}.  In Section \ref{Main-results}, we give our main results and some corollaries, propositions and examples as consequences.

\section{Preliminaries.} \label{prel}
Because we will need certain notions later, we  recall in this section some classical notions around convex cones, their duals, the barrier cone, recession cone, etc. On the other hand, we will recall the notion of Gateaux-differentiability, equi-Gateaux-differentiability, equi-lower semicontinuity,  etc. 

\subsection{The dual convex cone in a Hausdorff locally convex space.}
Let $Y$ be an Hausdorff locally convex space and $Y^*$ its topological dual. By  $\textnormal{int}(A)$, we denote the interior of a subet $A$ of $Y$.  By $\overline{\textnormal{conv}}^{w^*}(B)$ we denote the $w^*$-closed convex hull of a set $B\subset Y^*$. By $A^*$ we denote the dual positive cone of $A\subset Y$, defined by 
$$A^*:=\lbrace y^*\in Y^*: y^*(y)\geq 0; \hspace{1mm} \forall y\in A\rbrace.$$
The negative polar cone of $A$ is denoted $A^\circ:=-A^*$. We define the bidual cone of a set $A$ by 
$$A^{**}:=\lbrace y\in Y: y^*(y)\geq 0; \hspace{1mm} \forall y^*\in A^*\rbrace.$$
Recall that we have $A \subset A^{**}=\overline{\textnormal{cone}}(A)$ (the closed conique hull of $A$) and that $A^*=\overline{\textnormal{cone}}(A)^*$.

Let $K$ be a closed convex subset of $Y$, the barrier cone of $K$ is the closed convex cone defined as follows
$$\textnormal{bar}(K):=\lbrace y^* \in Y^*: \sup_{y\in K} y^*(y)<+\infty \rbrace.$$
 We denote $\mathcal{R}_K$  for the recession cone of the closed convex set $K$, that is, 
$$\mathcal{R}_K=\lbrace v\in Y: \forall \lambda >0, \forall x\in K, x+\lambda v\in K\rbrace.$$
It is well known (see for instance \cite{Za}, Ex. 2.45) that the $w^*$-closure of the barrier cone of $K$ is the polar of the recession cone of $K$, 
\begin{eqnarray}\label{theformulas}
\overline{\textnormal{bar}(K)}^{w^*}=(\mathcal{R}_K)^{\circ}=-(\mathcal{R}_K)^*.
\end{eqnarray}
Notice that $\textnormal{bar}(K)=-K^*$ and $\mathcal{R}_K=K$, if $K$ is a closed convex cone. We recall the following known consequence of the Hahn-Banach theorem.
\begin{proposition} \label{barrier} Let $K$ be a closed convex subset of $Y$. Then, we have 
\begin{eqnarray*}
 K&=&\cap_{y^*\in \textnormal{bar}(K)} \lbrace y \in Y : y^*(y)\leq \sup_{z\in K} y^*(z) \rbrace.
\end{eqnarray*}
In particular, $K=Y$ if and only if $\textnormal{bar}(K)=\lbrace 0 \rbrace$.
\end{proposition}
\begin{proof} Set $L:=\cap_{y^*\in \textnormal{bar}(K)} \lbrace y \in Y : y^*(y)\leq \sup_{z\in K} y^*(z) \rbrace$ and let us prove that $K=L$. Clearly, we have that $ K\subset L$. Suppose that $y_0\not \in K$. By the Hahn-Banach theorem, there exists $y^*_0\in Y^*\setminus \lbrace 0 \rbrace$ and $r\in \R$ such that 
\begin{eqnarray*}
y^*_0 (y_0)\geq r > \sup \lbrace y^*_0(y): y\in K \rbrace.
\end{eqnarray*}
It follows that  $y^*_0\in \textnormal{bar}(K)$ and $y_0\not \in L$. Thus, $L\subset K$.
\end{proof}

\subsection{Equi-Gateaux-differentiability and equi-semicontinuity.} \label{Equi-G} Let $E$ be an l.c.t.v. space, $\Omega$ be a nonempty open subset of $E$ and $(Y,\|\cdot\|)$ be a normed space.  Let $g : E \to Y$ be a function. We say that $g$ is Gateaux differentiable at $x\in \Omega$ if there exists a  linear and continuous map  $d_G g (x):E\to Y$ called the Gateaux-differential of $g$ at $x\in \Omega$ satisfying: $\forall v\in E$
$$\lim _{t \searrow 0} \left \|\frac {g(x+t v )-g(x)-td_G g (x)(v)}{t} \right \|=0.$$
If $E$ is a normed space, we denote $B_E(x,r)$ the closed ball centered at $x$ with radius $r>0$. We say that $g$ is Fr\'echet differentiable at $x\in \Omega$ if there exists a linear and continuous map  $g' (x):E\to Y$ called the Fr\'echet-differential of $g$ at $x\in \Omega$ (denoted also by $d_F g(x)$) such that: 
$$\lim _{t \searrow 0}\sup_{v\in B_E(0,1)} \left \|\frac {g(x+t v )-g(x)-tg' (x)(v)}{t} \right \|=0.$$
By $\langle\cdot, \cdot \rangle$ we denote the duality pairing between $E^*$ and $E$. A family $C$ of  functions from $E$ into $\R$ is  said to be equi-Gateaux differentiable (in short, equi-G-differentiable) at a point $x\in E$ if for every $\phi \in C$, $\phi : E \to \R$ is Gateaux differentiable at  $x$ and  for every $v \in E$ 
\begin{eqnarray*}
\lim_{t \searrow 0}\sup_{\phi \in C} \left |\frac{\phi(x+ t v) - \phi(x)- t \langle d_G \phi(x), v \rangle}{t} \right |=0.
\end{eqnarray*}
If $E$ is a normed space, we say that $C$ is equi-Fr\'echet differentiable (in short, equi-F-differentiable) at $x\in E$ if 
\begin{eqnarray*}
\lim_{t \searrow 0}\sup_{v\in B_E(0,1)}\sup_{\phi\in C} \left | \frac{\phi(x+tv)-\phi(x)-t\langle \phi'(x), v \rangle}{t}\right | =0,
\end{eqnarray*}
where $\phi'(x):=d_F \phi(x)$ denotes the Fr\'echet-differential of $\phi$ at $x$.

We say that  $C$ is equi-lower semicontinuous (in short equi-lsc) at $x\in E$, if for every $\varepsilon>0$ there exists an open  neighbourhood $\mathcal{O}_{x,\varepsilon}$ of $x$ (depending only on $x$ and $\varepsilon$) such that $\phi(y)-\phi(x)>-\varepsilon$, for all $y\in \mathcal{O}_{x,\varepsilon}$ and for all $\phi \in C$. If $E$ is a normed space, we say that $C$ is $r$-equi-Lipschitz at $x$ if there exists a ball centered at $x$ on which every function from $C$ is $r$-Lipschitz ($r\geq0$).

\begin{Exemp} \label{Exemp1} A basic and elementary examples are the following:

$(i)$ A finite family of Gateaux differentiable (resp. lower semicontinuous) functions from $E$ into $\R$ at some point, is equi-G-differentiable (resp. equi-lsc) at this point.

$(ii)$ Every nonempty subset $C$ of $E^*$ is equi-G-differentiable at every point. The set of $1$-Lipschitz maps in a normed space is a classical example of uniformly equi-continuous, hence equi-lsc functions at every point.

\end{Exemp}
We give in the following proposition some general examples of everywhere equi-F-differentiable functions. Let $X$ be a Banach space.  We denote $C^{1,\alpha}_b(X)$ the space of all real-valued, bounded, Lipschitz and continuously Fr\'echet differentiable functions $f$ on $X$  such that the Fr\'echet-differential $f'(\cdot):=d_F f(\cdot)$ is $\alpha$-Holder ($0<\alpha \leq 1$), that is, the space of all  continuously Fr\'echet differentiable functions $f$  such that $$\|f\|_{1,\alpha}:=\max(\|f\|_{\infty}, \|f\|_L, \|f\|_\alpha) <+\infty,$$ for all $f\in C^{1,\alpha}_b(X)$, where $\|f\|_{\infty}$ denotes the sup-norm of $f$ and
\begin{eqnarray*}
 \|f\|_L:=\sup \lbrace \frac{\|f(x_1)-f(x_2)\|}{\|x_1-x_2\|}: x_1,x_2\in X; x_1\neq x_2\rbrace.
\end{eqnarray*}
\begin{eqnarray*}
 \|f\|_\alpha:=\sup \lbrace \frac{\|f'(x_1)-f'(x_2)\|}{\|x_1-x_2\|^{\alpha}}: x_1,x_2\in X; x_1\neq x_2\rbrace.
\end{eqnarray*}
The space $(C^{1,\alpha}_b(X), \|\cdot\|_{1,\alpha})$ is a Banach space. We also need to introduce the space $C_b^G(Y)$ of all bounded, Lipschitz, G\^ateaux-differentiable functions $f$ from $X$ into $\R$ equipped with the norm
$\|f\|_G=\max(\|f\|_{\infty},\|d_G f\|_{\infty})$. Recall that by the mean value theorem, we have for every $f\in C_b^G(X)$ that
\[\|d_G f\|_{\infty}=\sup_{x,\overline{x}\in X; x\neq \overline{x}}\frac{|f(x)-f(\overline{x})|}{\|x-\overline{x}\|}(=:\|f\|_{L}).\]
The space $C_b^G(X)$, endowed with the mentioned norm, is a Banach space (see, \cite{DGZ}).

\begin{proposition} \label{EFrech} Let $X$ be a Banach space. Every nonempty bounded  subset $C \subset C^{1,\alpha}_b(X)$ is (everywhere) equi-F-differentiable and $r$-Lipschitz form some $r\geq 0$. Moreover, for all $x\in X$, the set $\overline{\textnormal{conv}}^{w^*}\lbrace f'(x): f\in C\rbrace$ is a $w^*$-compact subset of $X^*$.
\end{proposition}

\begin{proof} Let $r>0$ be such that $\|f\|_{1,\alpha} \leq  r$, for all $f\in C$. Clearly, $C$ is $r$-Lipschitz. Let $x_0, x\in X$, $t>0$  and $f\in C$. By the mean value theorem applied to the function $\xi \mapsto f(\xi)-\langle f'(x_0), \xi  \rangle$ on the intervalle $[x_0, x_0+ tx]$, there exists $\theta \in (0,1)$ such that 
\begin{eqnarray*}
f(x_0+tx)-f(x_0)-t\langle f'(x_0), x \rangle=\langle f'(x_0+\theta tx)-f'(x_0),tx \rangle.
\end{eqnarray*}
It follows that
\begin{eqnarray*}
|f(x_0+tx)-f(x_0)- t \langle f'(x_0), x \rangle| &\leq& \| f'(x_0+\theta tx)-f'(x_0)\|\|tx\|\\
&\leq& \|f\|_{\alpha}\|tx\|^{1+\alpha}\theta^\alpha\\
&\leq& r\|tx\|^{1+\alpha}.
\end{eqnarray*}
Then, we have 
\begin{eqnarray*}
\lim_{t \searrow 0}\sup_{x\in B_X(0,1)}\sup_{f\in C} \frac{|f(x_0+tx)-f(x_0)-t\langle f'(x_0), x \rangle|}{t} &\leq& \lim_{t \searrow 0} rt^\alpha =0,
\end{eqnarray*}
that is, $C$ is equi-F-differentiable at $x_0$. On the other hand, clearly we have $\overline{\textnormal{conv}}^{w^*}\lbrace f'(x): f\in C\rbrace \subset B_{X^*}(0,r)$, so it  is $w^*$-compact.
\end{proof}

\section{Admissible sets and the set of multipliers.} \label{Adm}

In this section we introduce the notions of {\it weak-admissible} and {\it admissible} sets which will allow us to generalize closed convex sets whose recession cones have nonempty interiors, this generalize in particular the class of closed convex cones with nonempty interiors.  Let $E$ be a l.c.t.v space. To each nonempty family $C$ of functions from $E$ into $\R$, we associated a subset of $E$, denoted $[C]^{\times}$,  as follows

$$[C]^{\times}:=\cap_{\phi\in C} \lbrace x\in E: \phi(x)\geq 0\rbrace.$$
Such sets are called in the literature $\Phi$-convex subsets of $E$ (A notion introduced by Ky Fan, see for instance \cite{KF1}). As an  immediate consequence of the Hahn-Banach theorem, we see that every closed convex subset of $E$ is a $\Phi$-convex subset, that is, every closed convex subset of $E$ is of the form $[C]^{\times}$ for some subset $C$ of $E^*+\R$.

\subsection{Weak-admissible and admissible sets.}

We give the following definitions related to the admissibility of a set at one of its points.

\begin{definition} \label{def1} Let $E$ be a l.c.t.v space. We say that $F\subset E$ is {\it weak-admissible} at $\hat{x}\in F$ if there exists  a nonempty family $C$ of real-valued functions $\phi: E \to \R$ such that:

$(a)$ $F=[C]^\times$.

$(b)$ $C $ is equi-G-differentiable at $\hat{x}$.

$(c)$ $S:=\lbrace \phi \in C: \phi(\hat{x})\neq 0 \rbrace$ is either empty or equi-lsc at $\hat{x}$.

$(d)$ the convex  set $\overline{\textnormal{conv}}^{w^*}\lbrace d_G \phi(\hat{x}): \phi\in C\rbrace $  is  $w^*$-compact  in $E^*$.

\noindent In this case, we say that $F$ is determined by $C$. We say that $F$ is weak-admissible if $F=[C]^\times$ is weak-admissible at each of its points.
\end{definition}

\begin{definition} \label{def2}
If $E$ is a normed space, we say that $F\subset E$ is {\it admissible} at $\hat{x}\in F$ if there exists a nonempty family $C$ of real-valued functions $\phi: E \to \R$ such that:

$(a)$ $F=[C]^\times$.

$(b)$ $C $ is equi-G-differentiable and $r$-equi-Lipschitz at $\hat{x}$.

$(c)$  $0 \not \in \overline{\textnormal{conv}}^{w^*}\lbrace d_G \phi(\hat{x}): \phi\in C\rbrace $ $($$\subset B_{E^*}(0,r)$$)$. 

\noindent We say that $F$ is an admissible set if $F=[C]^\times$ is admissible at each of its points. 
\end{definition}
Clearly, an admissible set is weak-admissible but the converse is not true in general.
The problem $\mathcal{P}$ is without constraints if the optimal solution $\hat{x}$ belongs to the interior of $F$. Thus, our work is of particular interest when $\hat{x}\in F\setminus \textnormal{int}(F)$. In this case we have the following proposition.

\begin{proposition} \label{lemmainf} Let $E$ be a l.c.t.v  space, $F\subset E$ be a  weak-admissible set at $\hat{x}\in F=[C]^\times$, determined by $C$. Suppose that  $\hat{x}\in F\setminus \textnormal{int}(F)$, then $\inf_{\xi \in C}\xi(\hat{x})= 0$. 
\end{proposition}

\begin{proof} Recall that since $F$ is weak-admissible at $\hat{x}$ determined by  $C$, then  $S:=\lbrace \phi \in C: \phi(\hat{x})\neq 0 \rbrace$ is either empty or equi-lsc at $\hat{x}$. It follows that if $\inf_{\xi \in C}\xi(\hat{x})\neq 0$ (hence $\inf_{\xi \in C}\xi(\hat{x})> 0$), then $S=C$ and so by the equi-lower semicontinuity of $S$ at $\hat{x}$, for some $\varepsilon \in (0, \inf_{\xi \in C}\xi(\hat{x}))$,  there exists an open neighbourhood $\mathcal{O}_{\hat{x},\varepsilon}$ of $\hat{x}$ such that $\xi(x)-\xi(\hat{x})>-\varepsilon$ on $\mathcal{O}_{\hat{x},\varepsilon}$ for all $\xi \in C$. Thus, $\xi(x)>\xi(\hat{x})-\varepsilon\geq \inf_{\xi \in C}\xi(\hat{x}) -\varepsilon \geq0$ for all $x\in\mathcal{O}_{\hat{x},\varepsilon}$ and all $\xi \in C$. In other words, $\hat{x}\in \textnormal{int}(F)$.
\end{proof}

The Banach  space $(C^{1,\alpha}(Y),\|\cdot\|_{1,\alpha})$, where $Y$ is a Banach space, gives a quite general framework to  build examples of both convex and non-convex admissible sets.  

\begin{Exemp} \label{ex1} Let $D$ be a bounded subset of $(C^{1,\alpha}(Y),\|\cdot\|_{1,\alpha})$  and let $p\in Y^*$ such that $\|p\|> r:=\sup_{\phi\in D}\|\phi'\|_{\infty}=\sup_{\phi\in D}\|\phi\|_L$. Then, the set 
$$A:=[D-p]^{\times}:=\lbrace y \in Y: (\phi- p)(y) \geq 0, \hspace{1mm} \forall \phi \in D\rbrace,$$
is an admissible set. Indeed, $C:=D-p$ is equi-G-differentiable by Proposition ~\ref{EFrech}. Moreover, each function from $C$ is $k$-Lipschitz with a same $k>0$ (we can take $k=r+\|p\|$) and  $\overline{\textnormal{conv}}^{w^*}\lbrace \phi'(y): \phi \in D\rbrace\subset B_{Y^*}(0,r)$ is $w^*$-compact. Finally, since $\|p\|>\sup_{\phi\in D}\|\phi'\|_{\infty}$, it follows that $p\not \in \overline{\textnormal{conv}}^{w^*}\lbrace \phi'(y): \phi \in D\rbrace$. In consequence, $0\not \in \overline{\textnormal{conv}}^{w^*}\lbrace \psi'(y): \psi \in C\rbrace$. 
\end{Exemp}

 More generally, we have the following proposition.  For each fixed point $y \in Y$, consider the following surjective  bounded  linear operator
\begin{eqnarray*}
\delta'_y: (C^{1,\alpha}(Y),\|\cdot\|_{1,\alpha}) &\to& (Y^*,\|\cdot\|)\\
                      \psi &\mapsto& \delta'_y(\psi):=d_G \psi(y).
\end{eqnarray*}

\begin{proposition} Let $y \in Y$ and $K$ be a  convex $w^*$-compact subset of $Y^*$ such that $0\not \in K$. Then, for every bounded subset $C$ of $C^{1,\alpha}(Y)$ such that $ \delta'_y(C)\subset K$, we have that the set $[C]^{\times}$ is a admissible at $y$. 
\end{proposition}
\begin{proof} We use the definition of admissible set together with Proposition ~\ref{EFrech}. 
\end{proof}
Notice that if  $K$ has a nonempty norm-interior, then by the continuity of $\delta'_y$, we see that $(\delta'_y)^{-1}(K)$ has also a nonempty interior in $(C^{1,\alpha}(Y),\|\cdot\|_{1,\alpha})$.
\vskip5mm
Now, we consider a class of closed convex admissible sets. Let $C$ be a bounded subset of $Y^*$ and $(\lambda_{y^*})_{y^*\in C}\subset \R$, then the set $$A= \lbrace y \in Y : y^*(y)-\lambda_{y^*}\geq 0; \forall y^*\in C\rbrace$$ is a closed convex, weak-admissible set at each of its points and is determined by the set  $\lbrace y^*-\lambda_{y^*}: y^* \in C\rbrace$. If moreover we assume that $0\not \in \overline{\textnormal{conv}}^{w^*} C$, then $A$ is admissible at each of its points. We prove in Theorem \ref{Mi1} below  that the class of admissible sets includes in particular convex closed sets $A\neq Y$ whose recession cone $\mathcal{R}_A$ has a nonempty interior, this includes in particular the classe of closed convex cones (different from $Y$) with a nonempty interior. We need the following  lemma. 

\begin{lemma} \label{Mi} Let $Y$ be a  real normed space. Then the following assertions hold.

$(i)$ Let  $A\neq Y$ be a nonempty closed convex cone  of $Y$ (equivalently $A^*\neq \lbrace 0 \rbrace$) and $e\in Y$. Then,  $e\in \textnormal{int}(A)$ if and only if  $\inf \lbrace y^*(e): y^*\in S_{Y^*} \cap A^*\rbrace>0$. In consequence,  $\textnormal{int}(A)\neq \emptyset$ if and only if,  either $A=Y$ (equivalently $A^*= \lbrace 0 \rbrace$) or $ 0\not \in \overline{\textnormal{conv}}^{w^*}(S_{Y^*} \cap A^*)$.

$(ii)$ Let $K$ be a nonempty closed convex set. Then, $\textnormal{int}(\mathcal{R}_K)\neq \emptyset$ if and only if either $K=Y$ or $ 0\not \in \overline{\textnormal{conv}}^{w^*}(S_{Y^*} \cap (-\textnormal{bar}(K))$.
\end{lemma}
\begin{proof} $(i)$ Suppose that $\inf \lbrace y^*(e): y^*\in S_{Y^*} \cap A^*\rbrace=\alpha>0$.  For all $h\in B_Y(0,\frac{\alpha}{2})$ and for all $ y^*\in S_{Y^*} \cap A^*$, we have $y^*(h+e)\geq -\|h\| +y^*(e) \geq -\|h\| + \alpha\geq \frac{\alpha}{2}$. In consequence, by positive homogeneity,   we have $z^*(h+e)\geq 0$  for all $z^*\in A^*$. It follows that $h+e\in A$ for all $h\in B_Y(0,\frac{\alpha}{2})$. In other words, $e\in \textnormal{int}(A)$. To see the converse, let $e\in \textnormal{int}(A)$.  Recall  that,
$$\textnormal{int}(A)\subset \lbrace y\in Y: y^*(y)> 0; \hspace{1mm} \forall y^*\in A^*\setminus \lbrace 0\rbrace\rbrace.$$
 Suppose by contradiction that the inclusion is not true, that is, there exists $a\in \textnormal{int}(A)$ and  some $y^*_0\in A^*\setminus \lbrace 0\rbrace$ satisfying $y^*_0(a)=0$. There exists $\varepsilon >0$ such that $B_Y(a,\varepsilon)\subset A$. Let $h\in B_Y(0,1)$, then we have $y^*_0(\varepsilon h)=y^*_0(-a)+y^*_0(a+\varepsilon h)=y^*_0(a+\varepsilon h)\geq 0$ since $a+\varepsilon h \in A$ and $y^*_0 \in  A^* \setminus \lbrace 0\rbrace$. This implies that $y^*_0(\varepsilon h)=0$ for all $h\in B_Y(0,1)$. Thus, $y^*_0=0$ which is a contradiction. Thus, our inclusion is true. Now, since $e\in \textnormal{int}(A)$,  then  $B_Y(e,\varepsilon)\subset \textnormal{int}(A)$ for some $\varepsilon >0$.  It follows that $$B_Y(e,\varepsilon)\subset \lbrace y\in Y: y^*(y)> 0; \hspace{1mm} \forall y^*\in A^* \setminus \lbrace 0\rbrace\rbrace.$$ 
Thus, for every $h\in B_Y(0,1)$ and every $y^*\in A^* \setminus \lbrace 0\rbrace$, we have that $y^*(e-\varepsilon h) >0$. Equivalently, $y^*(e)>\varepsilon y^*(h)$ for all $y^*\in S_{Y^*} \cap A^*$. By taking the supremum over $h\in B_Y(0,1)$, since $\|y^*\|=1$, we get $y^*(e) \geq \varepsilon$ for all $y^*\in S_{Y^*} \cap A^*$.  It follows that 
$$\inf_{y^*\in S_{Y^*} \cap A^*} y^*(e)\geq \varepsilon.$$
This completes the proof of the equivalence. To finish the part $(i)$,  it suffices to see that for each $e\in A$, $\inf_{y^*\in S_{Y^*} \cap A^*} y^*(e) = \inf_{y^*\in \overline{\textnormal{conv}}^{w^*}(S_{Y^*} \cap A^*)} y^*(e)$ and that by the Hahn-Banach theorem  $0\not \in \overline{\textnormal{conv}}^{w^*}(S_{Y^*} \cap A^*)$  if and only if there exists $e\in A\setminus \lbrace 0\rbrace$ such that $\inf_{y^*\in \overline{\textnormal{conv}}^{w^*}(S_{Y^*} \cap A^*)} y^*(e) >0$.
To prove $(ii)$, we apply part $(i)$ with the closed convex cone $A:=\mathcal{R}_K$, using the formula in (\ref{theformulas}).
\end{proof}

\begin{theorem} \label{Mi1} Let $Y$ be a normed vector space and $A$ be a  closed convex set $A$ such that $\textnormal{bar}(A)\neq \lbrace 0\rbrace$ (equivalently $A\neq Y$). Then, $A$ is weak-admissible at each of its points and is determined by the set $$C=\lbrace y^* - \inf_{x\in A} y^*(x): y^*\in S_{Y^*}\cap (-\textnormal{bar}(A))\rbrace.$$
Moreover, $A$ is an admissible set (everywhere, and determined by $C$) whenever $0\not \in \overline{\textnormal{conv}}^{w^*}(S_{Y^*} \cap (-\textnormal{bar}(A)))$. In consequence,  the following assertions hold.

$(i)$ Every closed convex set $A$ such that $A\neq Y$ and $\textnormal{int}(\mathcal{R}_A)\neq \emptyset$ is admissible (at each of its points), determined by $C$. 

$(ii)$ In particular, every closed convex cone $A$ of $Y$ with $A\neq Y$ and $\textnormal{int}(A)\neq \emptyset$ is an admissible set determined by $C=S_{Y^*} \cap A^*$. 
\end{theorem}

\begin{proof} From Proposition \ref{barrier}, we have 
\begin{eqnarray*}
 A &=& \cap_{y^*\in \textnormal{bar}(A)} \lbrace y \in Y : y^*(y)\leq \sup_{z\in A} y^*(z) \rbrace\\
    &=&  \lbrace y \in Y : y^*(y) - \inf_{x\in A} y^*(x)\geq 0, \hspace{1mm} \forall y^*\in  -\textnormal{bar}(A)\rbrace\\
   &=&  \lbrace y \in Y : y^*(y) - \inf_{x\in A} y^*(x)\geq 0, \hspace{1mm} \forall y^*\in  S_{Y^*} \cap (-\textnormal{bar}(A))\rbrace\\
&=& [C]^{\times},
\end{eqnarray*}
where,  $C=\lbrace y^* - \inf_{x\in A} y^*(x): y^*\in S_{Y^*}\cap (-\textnormal{bar}(A))\rbrace$. Clearly, $C$ is everywhere equi-G-differentiable family of $1$-Lipschitz functionals so $A$ is weak-admissible at each of its points by Definition \ref{def1}. If we assume that $0\not \in \overline{\textnormal{conv}}^{w^*}(S_{Y^*} \cap (-\textnormal{bar}(A)))$ then, clearly $$0\not \in \overline{\textnormal{conv}}^{w^*}\lbrace d_G \phi(\hat{x})): \phi \in C\rbrace = \overline{\textnormal{conv}}^{w^*}(S_{Y^*}\cap (-\textnormal{bar}(A))),$$
and so $A$ is an admissible set determined by $C$ (Definition \ref{def2}). The parts $(i)$ and $(ii)$ are consequences of  Lemma \ref{Mi}.

\end{proof}

\subsection{The set of multipliers.}

We introduce the following subset of the dual $E^*$. If $\hat{x} \in [C]^\times$ for some set $C$ of Gateaux differentiable functions at $\hat{x}$, we denote
$$\mathcal{T}_{C}(\hat{x}):=\cap_{n\geq 1} \overline{\textnormal{conv}}^{w^*}\lbrace d_G\phi(\hat{x}): \phi(\hat{x}) \in [0, \frac{1}{n}], \phi\in C\rbrace \subset E^*.$$
This set plays a crucial role in obtaining non-trivial Lagrange multipliers. In this paper, the set of non-trivial multipliers  associated to the prolem $(\mathcal{P})$, will be given from the set $\R^+\times \R^+\mathcal{T}_{C}(\hat{x})$. It is clear that $$\overline{\textnormal{conv}}^{w^*}\lbrace d_G\phi(\hat{x}): \phi(\hat{x}) =0, \phi\in C\rbrace \subset \mathcal{T}_{C}(\hat{x})\subset \overline{\textnormal{conv}}^{w^*}\lbrace d_G \phi(\hat{x}): \phi\in C\rbrace.$$ However, in the general framework when $C$ is infinite, these inclusions can be strict even in $\R^2$ as shown in the following example.

\begin{Exemp} In $E=\R^2$ let $C=\lbrace \phi_k:  k\geq 0 \rbrace \cup\lbrace \psi\rbrace$, where $\psi(x,y)=x+y$, $\phi_0(x,y)=x$,  and for all $k\geq 1$, $\phi_k(x,y)=y+\frac{1}{k}$. We see that $(0,0)\in [C]^\times=\R^+\times \R^+$, $\mathcal{T}_{C}((0,0))=\textnormal{conv}\lbrace (1,0), (0,1)\rbrace$, but $\textnormal{conv}\lbrace d_G\phi(\hat{x}): \phi(\hat{x}) =0, \phi\in C\rbrace = \lbrace (1,0)\rbrace$ and $\textnormal{conv}\lbrace d_G \phi(\hat{x}): \phi\in C\rbrace =\textnormal{conv}\lbrace (1,0), (0,1), (1,1)\rbrace $.
\end{Exemp}

\begin{remark} The set $\mathcal{T}_{C}(\hat{x})$ is of no interest when the optimal solution $\hat{x}\in \textnormal{int}([C]^\times)$ since in this case the problem $(\mathcal{P})$ is free of constraints. However, this set is crucial when $\hat{x}\in  [C]^\times\setminus \textnormal{int}([C]^\times)$. The following proposition guarantees that in this case, the set $\mathcal{T}_{C}(\hat{x})$ is always nonempty.
\end{remark}

\begin{proposition} \label{nonvide} Let $E$ be a l.c.t.v  space and $\hat{x}\in [C]^\times$ for some set $C$ of real-valued Gateaux differentiable functions at $\hat{x}$ such that $\overline{\textnormal{conv}}^{w^*}\lbrace d_G \phi(\hat{x}): \phi\in C\rbrace$ is $w^*$-compact in $E^*$. Then, $\mathcal{T}_{C}(\hat{x})\neq \emptyset$ if and only if $\inf_{\xi \in C}\xi(\hat{x})= 0$. In particular, $\mathcal{T}_{C}(\hat{x})\neq \emptyset$ whenever $\hat{x}\in F\setminus \textnormal{int}(F)$ where $F$ is a nonempty subset of $E$  weak-admissible at $\hat{x}$ and determined by $C$.
\end{proposition}
\begin{proof} Suppose that $\inf_{\xi \in C}\xi(\hat{x})= 0$. Then, for each $n\geq 1$, the set $D_n:=\overline{\textnormal{conv}}^{w^*}\lbrace d_G\phi(\hat{x}): \phi(\hat{x}) \in [0, \frac{1}{n}], \phi\in C\rbrace$ is nonempty and $w^*$-compact. Moreover, the sequence $(D_n)_{n\geq 1}$ is non-increasing and $\mathcal{T}_{C}(\hat{x})=\cap_{n\geq 1} D_n$. Hence, $\mathcal{T}_{C}(\hat{x})$ is a nonempty $w^*$-compact set as intersection of non-increasing sequence of nonemppty $w^*$-compact sets. The converse is trivial. Finally, if $\hat{x}\in F\setminus \textnormal{int}(F)$ where $F$ is a non empty subset of $E$, which is weak-admissible at $\hat{x}$ and determined by $C$, then using Proposition \ref{lemmainf} we get that $\inf_{\xi \in C}\xi(\hat{x})= 0$. Hence, $\mathcal{T}_{C}(\hat{x})\neq \emptyset$ using what has just been proved above.
\end{proof}

\subsection{Examples}
The examples below  and Example \ref{nonzero} shows that the set $\mathcal{T}_{C}(\hat{x})$ is precise enough to encompass the classical results found in the literature, such as when the set of constraints is finite or when $A=[C]^\times$ is a closed convex set.

\begin{Exemp} \label{ExN2} Let $E$ be a l.c.t.v space and $F=[C]^\times$ be a nonempty subset of $E$ weak-admissible at $\hat{x}\in F\setminus \textnormal{int}(F)$ and determined by $C$, so that $\mathcal{T}_{C}(\hat{x})\neq \emptyset$ by Proposition \ref{nonvide}. Then, the following assertions hold.

$(i)$ if $S:=\lbrace \phi \in C: \phi(\hat{x})\neq 0\rbrace$ is a  finite set, we have
$$\mathcal{T}_{C}(\hat{x})=\overline{\textnormal{conv}}^{w^*}\lbrace d_G\phi(\hat{x}): \phi(\hat{x})=0, \phi\in C\rbrace.$$
In particular if the set $C$ is finite, then 
$$\mathcal{T}_{C}(\hat{x})=\textnormal{conv}\lbrace d_G\phi(\hat{x}): \phi(\hat{x})=0, \phi\in C\rbrace.$$

$(ii)$ Assume that $E=Y$ is a Banach space and $C$ is relatively compact in $(C_b^G(X),\|\cdot\|_G)$. Then, 
$$\mathcal{T}_{C}(\hat{x})\subset \lbrace d_G\psi(\hat{x}): \psi(\hat{x})=0, \psi\in \overline{\textnormal{conv}}^{\|\cdot\|_G}(C)\rbrace.$$
If moreover,  $C$ is assumed to be convex and norm-compact, then
$$\mathcal{T}_{C}(\hat{x})= \lbrace d_G\psi(\hat{x}): \psi(\hat{x})=0, \psi\in C\rbrace.$$

$(iii)$ Assume that $E=Y$ is a normed vector space and $F\neq Y$ is  closed and convex, then $F=[C]^\times$, where $C=\lbrace y^* - \inf_{x\in F} y^*(x): y^*\in S_{Y^*}\cap (-\textnormal{bar}(F))\rbrace$ and we have  
\begin{eqnarray*}
\mathcal{T}_{C}(\hat{x})&\subset& \lbrace y^*\in \overline{\textnormal{conv}}^{w^*}(S_{Y^*}\cap (-\textnormal{bar}(F))) :  y^*(\hat{x})=\inf_{x\in F} y^*(x)\rbrace\\
&\subset& \lbrace y^*\in (\mathcal{R}_A)^*: y^*(\hat{x})=\inf_{x\in F} y^*(x)\rbrace.
\end{eqnarray*}
If $F\neq Y$ is a closed convex cone, then $F=[C]^\times$ where $C=S_{Y^*}\cap F^*$ and 
\begin{eqnarray*}
\mathcal{T}_{C}(\hat{x})&\subset& \lbrace y^*\in \overline{\textnormal{conv}}^{w^*}(S_{Y^*}\cap F^*) :  y^*(\hat{x})=0\rbrace\\
&\subset& \lbrace y^*\in  F^*:  y^*(\hat{x})=0\rbrace.
\end{eqnarray*}
\end{Exemp}

\begin{proof} The part $(i)$ is trivial. We prove $(ii)$. Indeed, we have 
\begin{eqnarray*}
\mathcal{T}_{C}(\hat{x})&:=&\cap_{n\geq 1} \overline{\textnormal{conv}}^{w^*}\lbrace d_G \phi(\hat{x}) :  \phi(\hat{x}) \in [0, \frac{1}{n}], \phi \in C\rbrace\\
&\subset & \cap_{n\geq 1} \overline{\lbrace d_G \psi(\hat{x}) :  \psi(\hat{x}) \in [0, \frac{1}{n}], \psi \in \overline{\textnormal{conv}}^{\|\cdot\|_G}(C) \rbrace}^{w^*}.
\end{eqnarray*}
The set $Q_n:=\lbrace  \psi \in \overline{\textnormal{conv}}^{\|\cdot\|_G}(C) :  \psi(\hat{x}) \in [0, \frac{1}{n}]\rbrace$ is compact as a closed subset of the compact $\overline{\textnormal{conv}}^{\|\cdot\|_G}(C)$ in the Banach space $(C_b^G(X),\|\cdot\|_G)$. Since the linear map $\delta'_{\hat{x}}: (C_b^G(X),\|\cdot\|_G)\to (Y^*,\|\cdot\|)$ defined by 
$\delta'_{\hat{x}}(\psi)=d_G \psi(\hat{x})$ is continuous, it follows that $\delta'_{\hat{x}}(Q_n)$ is norm compact in $Y^*$. In particular it is $w^*$-compact and so $w^*$-closed. Thus, we get
\begin{eqnarray*}
\mathcal{T}_{C}(\hat{x}) &\subset & \cap_{n\geq 1} \lbrace d_G \psi(\hat{x}) :  \psi(\hat{x}) \in [0, \frac{1}{n}], \psi \in \overline{\textnormal{conv}}^{\|\cdot\|_G}(C) \rbrace\\
&=&\lbrace d_G \psi(\hat{x}) :  \psi(\hat{x})=0, \psi \in \overline{\textnormal{conv}}^{\|\cdot\|_G}(C) \rbrace.
\end{eqnarray*}
If moreover we assume that $C$ is convex and compact, that is, $\overline{\textnormal{conv}}^{\|\cdot\|_G}(C)=C$, then clearly from the above inclusion we have $\mathcal{T}_{C}(\hat{x}) \subset  \lbrace d_G \psi(\hat{x}) :  \psi(\hat{x})=0, \psi \in C \rbrace.$ The reverse inclusion is always true. 

 In a similar way we prove $(iii)$. Indeed, by Theorem \ref{Mi1}, since $F$ is a closed convex set, it is determined by $C=\lbrace y^* - \inf_{x\in F} y^*(x): y^*\in S_{Y^*}\cap (-\textnormal{bar}(F))\rbrace.$ Thus, by definition
\begin{eqnarray*}
\mathcal{T}_{C}(\hat{x}) &:=& \cap_{n\geq 1} \overline{\textnormal{conv}}^{w^*}\lbrace d_G \phi(\hat{x}): \phi(\hat{x})\in [0,\frac{1}{n}], \phi \in C \rbrace\\
&=& \cap_{n\geq 1} \overline{\textnormal{conv}}^{w^*}\lbrace y^*\in S_{Y^*}\cap (-\textnormal{bar}(F)):  y^*(\hat{x}) -\inf_{x\in F} y^*(x) \in [0,\frac{1}{n}] \rbrace\\
&\subset&\cap_{n\geq 1} \lbrace y^* \in \overline{\textnormal{conv}}^{w^*}(S_{Y^*}\cap (-\textnormal{bar}(F))): y^*(\hat{x}) -\inf_{x\in F} y^*(x) \in [0,\frac{1}{n}]\rbrace\\
&=& \lbrace y^* \in \overline{\textnormal{conv}}^{w^*}(S_{Y^*}\cap (-\textnormal{bar}(F))): y^*(\hat{x})=\inf_{x\in F} y^*(x)\rbrace\\
&\subset& \lbrace y^*\in (\mathcal{R}_A)^*: y^*(\hat{x})=\inf_{x\in F} y^*(x)\rbrace \textnormal{ (by the formula (\ref{theformulas}))}
\end{eqnarray*}
If $F$ is a closed convex cone, we know that $F^*=-\textnormal{bar}(F)$ and $\inf_{x\in F} y^*(x)=0$, for all $y^*\in F^*$.
\end{proof}
Semi-infinite programming (SIP) problems are optimization problems in which there is an infinite number of variables or an infinite number of constraints (but not both). A general SIP problem can be formulated as
\begin{equation*}
(\mathcal{P})
\left \{
\begin{array}
[c]{l}
\max f\\
x \in \Omega \\
h(x,t)\geq 0, \forall t\in T
\end{array}
\right. 
\end{equation*}
where $\Omega$ is a nonempty open subset of $\R^p$, $x=(x_1,...,x_p)\in \R^p$, $T$ is an infinite set, and all the functions are real-valued. We denote $T(x)=\lbrace t\in T: h(x, t)=0\rbrace$, for $x\in \R^p$. 

\begin{Exemp} \label{ExN2bis} Let $T$ be a nonempty Hausdorff compact topological space, $\hat{x}\in \R^n$ and $h: \R^p\times T\to \R$ be a function such that:

$(i)$ for each $t\in T$, the function $h(\cdot, t)$ is Gateaux-differentiale at $\hat{x}$, we denote $\nabla_x h(\hat{x}, t)$ the Gateaux-differential of $h(\cdot, t)$ at $\hat{x}$,

$(ii)$ the functions $t\mapsto h(\hat{x}, t)$ and $t\mapsto \nabla_x h(\hat{x}, t)$ are continuous,

$(iii)$ $T(\hat{x})\neq \emptyset$.

\noindent Set $C=\lbrace h(\cdot,t): t\in T\rbrace$ and suppose that $\hat{x}\in [C]^\times$. Then, 
$$\mathcal{T}_{C}(\hat{x})=\textnormal{conv}\lbrace \nabla_x h(\hat{x}, t): t\in T(\hat{x})\rbrace.$$

\end{Exemp}
\begin{proof} First, notice that by using the continuity of the functions $h(\hat{x}, \cdot)$, $\nabla_x h(\hat{x}, \cdot)$ and the compactness of $T$, we get that $\nabla_x h(\hat{x}, T):=\lbrace \nabla_x h(\hat{x}, t): t\in T\rbrace$ is a compact subset of $\R^p$ and that 
$\lbrace \nabla_x h(\hat{x}, t):  h(\hat{x},t) \in [0, \frac{1}{n}], t\in T\rbrace$ is a closed subset of $\nabla_x h(\hat{x}, T)$ and so it is a compact subset, for each $n\geq 1$. It follows that $\textnormal{conv}\lbrace \nabla_x h(\hat{x}, t) :  h(\hat{x},t) \in [0, \frac{1}{n}], t\in T \rbrace$ is a compact subset of $\R^p$ for each $n\geq 1$. On the other hand, by the coincidence of the weak$^*$ and norm topologies on $\R^p$, using the definition of $\mathcal{T}_{C}(\hat{x})$, we get
\begin{eqnarray*}
\mathcal{T}_{C}(\hat{x}) &:=&\cap_{n\geq 1} \overline{\textnormal{conv}}^{w^*}\lbrace \nabla_x h(\hat{x}, t):  h(\hat{x},t) \in [0, \frac{1}{n}], t\in T\rbrace\\
&=& \cap_{n\geq 1} \textnormal{conv}\lbrace \nabla_x h(\hat{x}, t) :  h(\hat{x},t) \in [0, \frac{1}{n}], t\in T \rbrace\\
&\subset& \cap_{n\geq 1} \lbrace d_G \phi(\hat{x}) :  \phi(\hat{x}) \in [0, \frac{1}{n}], \phi \in  \textnormal{conv}\lbrace h(\hat{x}, t), t\in T \rbrace \rbrace\\
&=& \lbrace d_G \phi(\hat{x}) :  \phi(\hat{x})=0, \phi \in  \textnormal{conv}\lbrace h(\hat{x}, t), t\in T \rbrace \rbrace.
\end{eqnarray*}
Now, we prove that $$\lbrace d_G \phi(\hat{x}) :  \phi(\hat{x})=0, \phi \in  \textnormal{conv}\lbrace h(\hat{x}, t), t\in T \rbrace \rbrace \subset \textnormal{conv} \lbrace d_G h(\hat{x},t) :  t\in T(\hat{x}) \rbrace.$$ Indeed, if $q\in \lbrace d_G \phi(\hat{x}) :  \phi(\hat{x})=0, \phi \in  \textnormal{conv}\lbrace h(\hat{x}, t), t\in T \rbrace \rbrace$, then there exists $t_1,...,t_m\in T$, $\lambda_1, ...,\lambda_m\geq 0$ such that $\sum_{i=1}^m \lambda_i=1$, $q=\sum_{i=1}^m \lambda_i d_G h(\hat{x},t_i)$ and $\sum_{i=1}^m \lambda_i h(\hat{x},t_i)=0$. Since, $h(\hat{x},t_i)\geq 0$ for all $1\leq i \leq m$, we have $\lambda_i h(\hat{x},t_i)=0$ for all $1\leq i \leq m$. Thus, $\lambda_i=0$ if $h(\hat{x},t_i)\neq 0$ and so $q\in \textnormal{conv} \lbrace d_G h(\hat{x},t) :  t\in T(\hat{x}) \rbrace$. Finally, we have $\mathcal{T}_{C}(\hat{x})  \subset \textnormal{conv} \lbrace d_G h(\hat{x},t) :  t\in T(\hat{x}) \rbrace$. The reverse inclusion is always true.
\end{proof}

\begin{Exemp}  \label{nonzero} Let $Y$ be a normed vector space and $A=[C]^\times$ be an admissible set at $\hat{x}\in A\setminus \textnormal{int}(A)$. Then, $0\not \in \mathcal{T}_{C}(\hat{x})\neq \emptyset$. In consequence, the following assertions hold.

$(i)$ If $A\neq Y$ is a  closed convex subset such that $\textnormal{int}(\mathcal{R}_A)\neq \emptyset$ then, with $C=\lbrace y^* - \inf_{x\in A} y^*(x): y^*\in S_{Y^*}\cap (-\textnormal{bar}(A))\rbrace$ we have $A=[C]^\times$, $A$ is admissible at $\hat{x}\in A\setminus \textnormal{int}(A)$, $0\not \in \mathcal{T}_{C}(\hat{x})\neq \emptyset$ and 
\begin{eqnarray*}
\mathcal{T}_{C}(\hat{x}) &\subset& \lbrace y^*\in \overline{\textnormal{conv}}^{w^*}(S_{Y^*}\cap (-\textnormal{bar}(A))) :  y^*(\hat{x})=\inf_{x\in A} y^*(x)\rbrace\\
&\subset& \lbrace y^*\in (\mathcal{R}_A)^*\setminus\lbrace 0\rbrace: y^*(\hat{x})=\inf_{x\in A} y^*(x)\rbrace.
\end{eqnarray*}

$(ii)$ In particular, if $A\neq Y$ is a  closed convex cone with a nonempty interior, then $A=[S_{Y^*}\cap A^*]^\times$ and $0\not \in \mathcal{T}_{S_{Y^*}\cap A^*}(\hat{x})\subset \lbrace y^*\in \overline{\textnormal{conv}}^{w^*}(S_{Y^*}\cap A^*) :  y^*(\hat{x})=0\rbrace\subset \lbrace y^*\in  A^*\setminus \lbrace 0\rbrace :  y^*(\hat{x})=0\rbrace.$ 
\end{Exemp}

\begin{proof} By Proposition \ref{nonvide}, $\mathcal{T}_{C}(\hat{x})\neq \emptyset$. By the definition of admissible set, $0\not \in \overline{\textnormal{conv}}^{w^*}\lbrace d_G\phi(\hat{x}): \phi\in C\rbrace$ and by the definition of $\mathcal{T}_{C}(\hat{x})$, we have $\mathcal{T}_{C}(\hat{x}) \subset \overline{\textnormal{conv}}^{w^*}\lbrace d_G\phi(\hat{x}): \phi\in C\rbrace$.
Hence, $0\not \in \mathcal{T}_{C}(\hat{x})\neq \emptyset$. For the rest, we use the part $(iii)$ of Example \ref{ExN2} and Theorem \ref{Mi1}.
\end{proof}

\section{The main results.} \label{Main-results}
This section is divided into two parts. We first treat the case of optimization problems with only constraints of inequalities in the general framework of l.c.t.v. spaces. Then, we will consider the case of problems with both inequality and equality constraints, in the framework of Banach spaces.

\subsection{Inequalities constraints in a Hausdorff locally convex space.}

Before giving Theorem \ref{FarkasMM}, we need some notation and reminders. Let $D$ be a  compact convex subset of  an l.c.t.v. space $E$.  We denote by $(\mathcal{C}(D),\|\cdot\|_{\infty})$ the Banach space of all real-valued continuous functions on $D$. By $\mathcal{C}^+(D)$ we denote the positive cone of $\mathcal{C}(D)$ which has a nonempty interior and $\mathcal{C}(D)^*$ denotes the topological dual of $\mathcal{C}(D)$. The space $\textnormal{Aff}(D)$ denotes the space of all affine continuous functions from $D$ to $\R$. We also recall (see \cite{Ru}) that the dual space $\mathcal{C}(D)^{\ast}$ is naturally identified with the Radon measures on $D$ via the duality map
\begin{eqnarray*} \label{eq:duality map}
\langle \mu,f\rangle=\int_{D}fd\mu,\quad \text{for all }\mu \in \mathcal{C}
(D)^{\ast}\text{ and }f\in \mathcal{C}(D). 
\end{eqnarray*}
In particular, the evaluation map $\delta_{x}: f\mapsto f(x)$ is the Dirac measure of $x\in D$ and we have:
\[
\delta_{x}(f):=\langle \delta_{x},f\rangle=f(x).
\]
Furthermore, the dual norm $||\mu||_{\ast}$ coincides with the total variation
of the measure $\mu$ denoted $\|\mu\|_{TV}.$\smallskip

We denote by $\mathcal{M}^{1}(D)$ the set of all Borel probability measures on
$D$. This set is a $w^{\ast}$-compact convex subset of $\mathcal{C}(D)^{\ast}$
and coincides with the weak$^{\ast}$ closed convex hull of the set
$\delta(D):=\lbrace \delta_x : x\in D\rbrace$, that is,
\begin{eqnarray} \label{eq:M1(D)}
\mathcal{M}^{1}(D)&=&\left \{  \mu \in \mathcal{C}(D)^{\ast}:\Vert \mu \Vert_{\ast}=\langle \mu,\boldsymbol{1}_D\rangle=1\right \}  \\
&=&\overline{\mathrm{conv}}^{w^{\ast}}\left(  \delta(D)\right),\nonumber
\end{eqnarray}
where $\boldsymbol{1}_D(x)=1,$ for all $x\in D.$

\vskip5mm
If $w\in E$ and $K\subset E$ is a nonempty set, we denote $[w,K]$ the convex hull of the point $w$ and the set $K$, that is, $[w,K]=\lbrace \lambda w +(1-\lambda)x: \lambda \in [0,1], x\in K \rbrace$.
\begin{lemma} \label{lemmacompact} Let $E$ be a l.c.t.v  space, $w\in E$ and $(K_n)_{n\geq1}$ be a non-increasing sequence of nonempty compact sets. Then, $\cap_{n\geq 1} [w, K_n]=[w, \cap_{n\geq 1} K_n]$.
\end{lemma}
\begin{proof} Clearly, $[w, \cap_{n\geq 1} K_n] \subset \cap_{n\geq 1} [w, K_n]$. Let $x\in \cap_{n\geq 1} [w, K_n]$, then for every $n\geq 1$ there exists $x_n \in K_n$ and $\lambda_n \in [0,1]$ such that $x=\lambda_n w+(1-\lambda_n) x_n$. Since $[0,1]$ is compact and $(K_n)_{n\geq1}$ is a non-increasing sequence of nonempty compact sets, there are subnets $(x_\alpha)$ and a $(\lambda_\alpha)$ converging respectively to some $\bar{x}\in \cap_{n\geq 1} K_n$ and $\bar{\lambda}\in [0,1]$. Thus, $x=\lim_{\alpha} (\lambda_\alpha w+(1-\lambda_\alpha) x_\alpha) =\bar{\lambda} w+(1-\bar{\lambda}) \bar{x}\in [w, \cap_{n\geq 1} K_n]$.
\end{proof}

Now, we give the proof of our first main result mentioned in the introduction.

\begin{theorem} \label{FarkasMM}  Let $E$ be a l.c.t.v space,  $\Omega$ be a nonempty open subset of $E$ and $F$ be a non empty subset of $E$. Let $f: \Omega \to \R$ be a function. Assume that  $\hat{x} \in \Omega$  is an optimal solution of the problem $(\mathcal{P})$:
\begin{equation*}
(\mathcal{P})
\left \{
\begin{array}
[c]{l}
\max f\\
x\in \Omega \\
x\in F
\end{array}
\right. 
\end{equation*}
that $f$ is Gateaux differentiable at $\hat{x}$ and that $F$ is weak-admissible at $\hat{x}$ and determined by $C$. Then, either $\hat{x}\in \textnormal{int}(F)$ and therefore  $d_G f(\hat{x})=0$, otherwize if $\hat{x}\in F\setminus \textnormal{int}(F)$, we have $0\in [d_G f(\hat{x}),\mathcal{T}_{C}(\hat{x})]$. That is, there exists $(\lambda^*,  \beta^*)\in  \R^+\times \R^+$  and  $x^*\in \mathcal{T}_{C}(\hat{x})$ such that

$(i)$ $(\lambda^*, \beta^*)\neq (0,0)$.

$(ii)$ $\lambda^*  d_G f(\hat{x}) + \beta^* x^*=0$.

If moreover, we assume that  $0\not \in \mathcal{T}_{C}(\hat{x})$ (in particular if $E$ is a normed space and $F$ is admissible at $\hat{x}$), then we can choose $\lambda^*=1$.
\end{theorem}

\begin{proof} Clearly, $d_G f(\hat{x})=0$ if $\hat{x}\in \textnormal{int}(F)$. Suppose that $\hat{x}\in F\setminus \textnormal{int}(F)$. Then, by Proposition \ref{lemmainf}, we have  $\inf_{\xi \in C}\xi(\hat{x})=0$.  For each integer number $n\geq 1$, set $S_n=\lbrace \phi \in C: \phi(\hat{x}) > \frac{1}{n}\rbrace$ (may be an empty set) and $C_n=C\setminus S_n=\lbrace \phi \in C: 0\leq \phi(\hat{x})\leq \frac{1}{n}\rbrace\neq \emptyset$ (since $\inf_{\xi \in C} \xi(\hat{x})=0$). Clearly,  the sequence $(C_n)_{n\geq 1}$ is non-increasing.  Let us set $D_n=\overline{\textnormal{conv}}^{w^*}\lbrace d_G \phi(\hat{x}): \phi\in C_n\rbrace $, which is nonempty and $w^*$-compact convex subset of $E^*$ (the sequence $(D_n)_{n\geq1}$ is non-increasing). Notice that for each $u\in E$,  the evaluation mapping $\delta_u: p \in D_n \mapsto \langle p, u \rangle$,  is linear and $w^*$-continuous, thus $\delta_u\in \textnormal{Aff}(D_n)\subset C(D_n)$, where $D_n$ is equipped with the $w^*$-topology. Let us set $$X:=\lbrace (\delta_u, d_G f(\hat{x})(u)) : u\in E \rbrace\subset \mathcal{C}(D_n)\times \R,$$ which is a vector subspace and let us prove that $$(-\boldsymbol{1}_{D_n},-1)\not \in \overline{X+\mathcal{C}^+(D_n)\times \R^+},$$ where $\boldsymbol{1}_{D_n}$ denotes the constant function equal to $1$ on $D_n$ and the closure is taken in the Banach space $\mathcal{C}(D_n)\times \R$.  Suppose by contradiction that $ (-\boldsymbol{1}_{D_n},-1) \in \overline{X+\mathcal{C}^+(D_n)\times \R^+}$.  There exists $u \in E$ and $(h_0,z_0)\in  \mathcal{C}^+(D_n)\times \R^+$ such that 
\begin{eqnarray*}
\|h_0 +\boldsymbol{1}_{D_n} - \delta_u\|_{\infty} +|z_0+1 - d_G f(\hat{x})(u)|<\frac{1}{2}
\end{eqnarray*}
 It follows  that, for all $p \in D_n=\overline{\textnormal{conv}}^{w^*}\lbrace d_G \phi(\hat{x})): \phi\in C_n\rbrace$
\begin{eqnarray} \label{eq2110}
\langle p , u\rangle > h_0(p)+\boldsymbol{1}_D(p) - \frac{1}{2} \geq  \frac{1}{2}
\end{eqnarray}
\begin{eqnarray} \label{eq3110}
d_G f(\hat{x})(u) > z_0+\frac{1}{2}\geq   \frac{1}{2}.
\end{eqnarray}
 By assumption, $f$ is Gateaux differentiable at $\hat{x}$ and the family $C$ is equi-Gateaux differentiable at $\hat{x}$, then there exists $\delta_0>0$ such that for all $\delta \in (0, \delta_0)$, we have $\hat{x}+ \delta u\in \Omega$ and 
\begin{eqnarray*} \label{eq212}
\sup_{\phi \in C}|\frac{\phi (\hat{x}+ \delta u) -\phi(\hat{x}) - \delta  \langle d_G \phi (\hat{x}), u\rangle}{\delta}| &<& \frac{1}{4} \\
|\frac{f (\hat{x} + \delta u) -f (\hat{x})- \delta d_G f (\hat{x})(u)}{\delta}|< \frac{1}{4}
\end{eqnarray*}
Using  (\ref{eq2110}), (\ref{eq3110}) and the two last inequalities, we get  that for all $\delta \in (0, \delta_0)$, $\hat{x}+\delta u \in \Omega$  and  
\begin{eqnarray*}
\forall \phi \in C_n \subset C, \hspace{3mm} \phi (\hat{x}+ \delta u) &>& \delta (\langle d_G \phi (\hat{x}), u\rangle  -\frac{1}{4}) + \phi(\hat{x})\\
                                                                                 &>& \delta \frac{1}{4}>0 
\end{eqnarray*}
\begin{eqnarray*}
f(\hat{x} + \delta u) -f (\hat{x}) &>& \delta (d_G f(\hat{x})(u) -\frac{1}{4})\\
                                                     &>& \delta \frac{1}{4}  >  0.
\end{eqnarray*}
It follows that, for all $\delta\in (0,\delta_0)$ and for all $\phi \in C_n$, we have 
\begin{eqnarray}\label{eqfin} 
\hat{x} + \delta u \in \Omega,  \phi (\hat{x} + \delta u) >0 \textnormal{ and } f(\hat{x}+\delta u) -f (\hat{x}) >0.
\end{eqnarray}
 If $S_n=\emptyset$, then $C_n=C$ and so (\ref{eqfin}) contradicts the fact that $\hat{x}$ is an optimal solution of $(\mathcal{P})$. We will show that there is also a contradiction in the case where $S_n\neq \emptyset$. Indeed, since the family $S=\lbrace \phi \in C: \phi(\hat{x})\neq 0 \rbrace$ is equi-lsc at $\hat{x}$ (see Definition \ref{def1}), the same applies to  $S_n\subset S$, so there exists $\alpha_n>0$ such that: $\forall \delta\in (0,\alpha_n), \forall \phi \in S_n$
\begin{eqnarray}\label{mod}
\hat{x}+\delta u\in \Omega \textnormal{ and } \phi(\hat{x}+\delta u) &>& \phi(\hat{x})  -\frac{1}{n}>0.
\end{eqnarray}
The formulas (\ref{eqfin}) and (\ref{mod}) applied with $\delta \in (0,\min(\delta_0,\alpha_n))$, contradict also the fact that $\hat{x}$ is an optimal solution of $(\mathcal{P})$. Finally, we proved that $ (-\boldsymbol{1}_{D_n},-1) \not \in  \overline{X+\mathcal{C}^+(D_n)\times \R^+}$. By the Hahn-Banach theorem, (and using the fact that $\overline{X+\mathcal{C}^+(D_n)\times \R^+}$ is a cone) there exists a Radon measure $\mu^*_n$ and $\lambda^*_n \in \R$ such that $(\mu^*_n,\lambda^*_n)\neq (0,0)$ and 
\begin{eqnarray*} \label{sup}
 \int_D h d\mu^*_n +  \lambda^*_n z \geq 0, \hspace{3mm} \forall (h,z)\in X+\mathcal{C}^+(D_n)\times \R^+.
\end{eqnarray*}
This implies in particular that $\mu^*_n\geq 0$, $\lambda^*_n\geq 0$ and for all $u\in E$,
\begin{eqnarray}\label{for}
\lambda^*_n  d_G f(\hat{x})(u)+  \int_{D_n}  \delta_u d\mu^*_n=0.
\end{eqnarray}

{\it Case 1.} If there exists some $n_0\geq 1$ such that $\mu^*_{n_0}=0$, then $\lambda^*_{n_0}\neq 0$ (since $(\mu^*_{n_0},\lambda^*_{n_0})\neq (0,0)$) and so from $(\ref{for})$, we get that $d_G f(\hat{x})=0$ and in consequence the theorem works with $(\lambda^*, \beta^*)=(1,0)$.

{\it Case 2.}  Suppose that $\mu^*_n\neq 0$ for all $n\geq 1$. Then $\nu^*_n:=\frac{\mu^*_n}{\|\mu^*_n\|_{TV}}$ is  a Borel probability measure on the $w^*$-compact convex set $D_n$ and from $(\ref{for})$, we have  
\begin{eqnarray}\label{for1}
\frac{\lambda^*_n}{\lambda^*_n +\|\mu^*_n\|_{TV}}  d_G f(\hat{x})(u)+  \frac{\|\mu^*_n\|_{TV}}{\lambda^*_n +\|\mu^*_n\|_{TV}} \int_{D_n}  \delta_u d\nu^*_n =0.
\end{eqnarray}
 As a Borel probability measure,  $\nu^*_n$ has a unique barycenter $p_n\in D_n=\overline{\textnormal{conv}}^{w^*}\lbrace d_G\phi(\hat{x}): \phi\in C_n\rbrace$  (see for instance \cite[Chapitre IV, section 7, n$^\circ$1, Corollaire de Proposition 1.]{Bour} or \cite[Lemma 10]{ChM}), that is, $k(p_n)=\int_{D_n} k d \nu^*_n$, for every $k\in \textnormal{Aff}(D_n)$. In particular, since for every $u\in E$ the map  $\delta_u\in \textnormal{Aff}(D_n)$, we obtain $\int_{D_n} \delta_u d\nu^*_n =\langle p_n, u\rangle$. Using (\ref{for1}) and the fact that $\int_{D_n} \delta_u d\nu^*_n =\langle p_n, u\rangle$ for all $u\in E$, we get, 
\begin{eqnarray*}
0\in [d_G f(\hat{x}), D_n], \forall n\geq 1.
\end{eqnarray*}
Hence, $0\in \cap_{n\geq 1} [d_G f(\hat{x}), D_n]=[d_G f(\hat{x}),\cap_{n\geq 1}D_n]$ by using Lemma \ref{lemmacompact}, since $(D_n)_{n\geq 1}$ is a non-increasing sequence of $w^*$-compact subsets. Equivalently, there exists $(\lambda^*,  \beta^*)\in  \R^+\times \R^+$  and $x^*\in \cap_{n\geq 1}D_n:=  \mathcal{T}_{C}(\hat{x})$ such that 

$(i)$ $(\lambda^*, \beta^*)\neq (0,0)$.

$(ii)$ $\lambda^*  d_G f(\hat{x}) +  \beta^* x^*=0 $.

\noindent Finally, we see from $(i)$ and $(ii)$ that if $\lambda^*=0$ then $\beta^* \neq 0$ and so $0 \in \mathcal{T}_{C}(\hat{x})$. Thus, if $0 \not \in \mathcal{T}_{C}(\hat{x})$ then $\lambda^*\neq 0$ and so, dividing by $\lambda^*$ we can assume that $\lambda^*=1$ (in particular in a normed space if $F$ is admissible at $\hat{x}$, we have by Example \ref{nonzero} that $0 \not \in \mathcal{T}_{C}(\hat{x})$). 
\end{proof}

We show in Example \ref{ExN1} below, that the set $\mathcal{T}_{C}(\hat{x})$ in Theorem \ref{FarkasMM} cannot, in general, be replaced by  the set $\overline{\textnormal{conv}}^{w^*}\lbrace d_G\phi(\hat{x}): \phi(\hat{x})=0, \phi\in C\rbrace $ even in $\R^2$. 

\begin{Exemp} \label{ExN1} Indeed, in $E=\R^2$, let $\Omega=]-1,1[\times ]-1,1[$, $C=\lbrace \phi_k:  k\geq 0 \rbrace$, where $\phi_0(x,y)=x$,  and for all $k\geq 1$, $\phi_k(x,y)=y+\frac{1}{k}$ and $f(x,y)= -x^2-y$ for all $(x,y)\in \Omega$. We set $F:=[C]^\times=\R^+\times \R^+$. In this case, the problem $(\mathcal{P})$ has $(0,0)$ as a solution, $F$ is clearly weak-admissible at $(0,0)\in F\setminus \textnormal{int}(F)$. We have $\phi_0(0,0)=0$,  $\phi_k(0,0)=\frac{1}{k}>0$ for all $k\geq 1$, $\phi'_0(0,0)=(1,0)$, $\phi'_k(0,0)=(0,1)$ for all $k\geq 1$ and $f'(0,0)=(0,-1)$. We see that $$\textnormal{conv}\lbrace \phi'(0,0): \phi(0,0)=0, \phi\in C\rbrace =\lbrace (1,0) \rbrace,$$ 
$$\mathcal{T}_{C}(\hat{x}):=\cap_{n\geq 1} \textnormal{conv}\lbrace \phi'(0,0): \phi(0,0)\in [0, \frac{1}{n}], \phi\in C\rbrace =[(1,0), (0,1)],$$ 
 $$(0,0) \not \in [(0,-1), (1,0) ]=[f'(0,0), \phi'_0(0,0)],$$ but clearly, $(0,0)\in [(0,-1), \mathcal{T}_{C}(\hat{x})]=[f'(0,0), \mathcal{T}_C(\hat{x})]$.
\end{Exemp}

\begin{remark}In a normed vector space, the condition of $w^*$-compactness of $\overline{\textnormal{conv}}^{w^*}\lbrace d_G \phi(\hat{x}): \phi\in C\rbrace $  in the definition of weak-admissibility of a set $F$ can  be omitted by replacing the set $C$ by the set $\widetilde{C}:= \lbrace \frac{\phi}{\max(1,\|d_G \phi(\hat{x})\|)}:\phi \in C \rbrace$ and $\mathcal{T}_{C}(\hat{x})$ by $\mathcal{T}_{\widetilde{C}}(\hat{x})$, since $C$ and $\widetilde{C}$ determine the same set $F$, that is, $F=[C]^\times =[\widetilde{C}]^\times$ and the set $\overline{\textnormal{conv}}^{w^*}\lbrace d_G \phi(\hat{x}): \phi\in \widetilde{C}\rbrace $ is $w^*$-compact.
\end{remark}

An application to semi-infinite programming (SIP) problems is given in the following corollary. Some general literature on semi-inﬁnite programming problems can be found in \cite{Cz, HK, Sh, RG, LS}. Recall that $\nabla_x h(\hat{x}, t)$ denotes the Gateaux-differential of $h(\cdot, t)$ at $\hat{x}$ and $T(\hat{x}):=\lbrace t\in T: h(\hat{x},t)=0\rbrace$.

\begin{corollary} \label{semi-finite} Let $T$ be a nonempty Hausdorff compact topological space and $h: \R^p\times T\to \R$, $f:\R^p\to \R$ be functions. Assume that $\hat{x}\in \R^p$ is an optimal solution of the problem:
\begin{equation*}
(\mathcal{P})
\left \{
\begin{array}
[c]{l}
\max f\\
x \in \Omega \\
h(x,t)\geq 0, \forall t\in T
\end{array}
\right. 
\end{equation*}
and that:

$(a)$ $T(\hat{x})\neq \emptyset$.

$(b)$ the function $f$ is Gateaux differentiable at $\hat{x}$ and family $(h(\cdot, t))_{t\in T}$ is equi-Gateaux-differentiale at $\hat{x}$, 

$(c)$ the family $(h(\cdot,t))_{t\in T\setminus T(\hat{x})}$ is equi-lsc at $\hat{x}$, 

$(d)$ the functions $t\mapsto h(\hat{x}, t)$ and $t\mapsto \nabla_x h(\hat{x}, t)$ are continuous.

\noindent Then,  there exists $\lambda_i\geq 0$, $t_i\in T(\hat{x})$, $i=0,...,k$ such that $k\leq p$, $\sum_{i=0}^k\lambda_i=1$ and 
$$\lambda_0 d_G f(\hat{x})+\sum_{i=1}^k\lambda_i \nabla_x h(\hat{x}, t_i)=0.$$
If moreover, we assume that  $0\not \in \textnormal{conv}\lbrace \nabla_x h(\hat{x}, t), t\in T(\hat{x})\rbrace$, then $\lambda_0\neq 0$.
\end{corollary}
\begin{proof} Set $F:=[C]^\times:=\lbrace x\in \R^p: h(x,t)\geq 0, \forall t\in T\rbrace$, where $C=\lbrace h(\cdot,t): t\in T\rbrace$. From our hypothesis, the set $F$ is weak-admissible at $\hat{x}$ and by Example \ref{ExN2bis}, $\mathcal{T}_{C}(\hat{x})=\textnormal{conv}\lbrace \nabla_x h(\hat{x}, t): t\in T(\hat{x})\rbrace.$ Using Theorem \ref{FarkasMM}, we have that either $d_G f(\hat{x})=0$ or $0\in [d_G f(\hat{x}),\mathcal{T}_{C}(\hat{x})]\subset \R^p$. If $d_G f(\hat{x})=0$, then the corollary works with $\lambda_0=1$ and $\lambda_i=0$ for $1\leq i\leq k$ and $k\leq p$. If $0\in [d_G f(\hat{x}),\mathcal{T}_{C}(\hat{x})]$, then by Carathéodory's theorem there are $\lambda_i\geq 0$, $t_i\in T(\hat{x})$, $i=0,...,k$ such that $k\leq p$, $\sum_{i=0}^k\lambda_i=1$ and 
$\lambda_0 d_G f(\hat{x})+\sum_{i=1}^k\lambda_i \nabla_x h(\hat{x}, t_i)=0$. Finally, it is clear that if $0\not \in \textnormal{conv}\lbrace \nabla_x h(\hat{x}, t), t\in T(\hat{x})\rbrace$, then $\lambda_0\neq 0$.
\end{proof}
\begin{corollary} \label{ExN5} Let $E=Y$ be a normed vector space and $C$ be a nonempty  relatively compact subset of $(C_b^{1\alpha} (X),\|\cdot\|_{1,\alpha})$.  Let $\Omega$ be a nonempty open subset of $Y$, $\hat{x}\in \Omega$ and $f: \Omega \to \R$ be a Gateaux differentiable function. Suppose that $\hat{x}$ is a solution of the problem
\begin{equation*}
(\mathcal{P})
\left \{
\begin{array}
[c]{l}
\max f\\
x\in \Omega \\
\phi(x)\geq 0, \forall \phi \in C.
\end{array}
\right. 
\end{equation*}
 Then, either $d_G f(\hat{x})=0$ or there exists  $(\lambda^*,  \beta^*)\in  \R^+\times \R^+$  and $\psi \in \overline{\textnormal{conv}}^{\|\cdot\|_{1,\alpha}}(C)$ such that,

$(i)$ $(\lambda^*, \beta^*)\neq (0,0)$.

$(ii)$ $\lambda^*  d_G f(\hat{x})+ \beta^* d_G\psi(\hat{x})=0$,

$(iii)$ $\psi(\hat{x})=0$.

\noindent If moreover, we assume that  $0\not \in \lbrace d_G \phi(\hat{x}): \phi(\hat{x})=0, \phi \in \overline{\textnormal{conv}}^{\|\cdot\|_{1,\alpha}}(C)\rbrace$, then we can take $\lambda^* =1$.
\end{corollary}
\begin{proof} Using Proposition \ref{EFrech}, we see that $F=[C]^\times$ is weak-admissible at each point, then we apply Theorem \ref{FarkasMM} and part $(ii)$ of Example \ref{ExN2}.
\end{proof}

In order to deal with problems of the form
\begin{equation*}
(\mathcal{P})
\left \{
\begin{array}
[c]{l}
\max f\\
x\in \Omega \\
g(x) \in A
\end{array}
\right. 
\end{equation*}
we need the following lemma.
\begin{lemma}\label{Gateaux} Let $E$ be an l.c.t.v.  space, $\Omega$ an open subset of $E$ and  $(Y,\|\cdot\|)$ be a normed space. Let  $g: \Omega \to Y$ be a Gateaux differentiable and continuous function at $\hat{x}\in \Omega$. Let $C$ be a family of  functions from $Y$ into $\R$, $r$-equi-Lipschitz ($r\geq 0$) at $g(\hat{x})$ and equi-Gateaux differentiable at $g(\hat{x})$. Then,

$(i)$ the set $\overline{\textnormal{conv}}^{w^*}\lbrace d_G\phi(g(\hat{x})): \phi\in C\rbrace$,  is a $w^*$-compact  subset of $Y^*$. 

$(ii)$  the family $\lbrace \phi\circ g: \phi \in C \rbrace$ is equi-Gateaux differentiable at $\hat{x}$  and $d_G(\phi\circ g)(\hat{x})=d_G\phi(g(\hat{x}))\circ d_G g(\hat{x})$ for all $\phi \in C$.
\end{lemma}
\begin{proof} Since $C$ is  equi-Gateaux differentiable at $g(\hat{x})$, then for  every $y \in Y$ and every $\varepsilon >0$, there exists $\delta>0$ such that for all $t\in ]0,\delta[$
\begin{eqnarray*}
\sup_{\phi \in C} \left |\frac{\phi(g(\hat{x})+ t y) - \phi(g(\hat{x}))- t d_G \phi(g(\hat{x}))(y)}{t} \right |<\varepsilon.
\end{eqnarray*}
With $y=d_G g(\hat{x})(u)\in Y$ for $u\in E$, we have that for  every $u \in E$ and every $\varepsilon >0$, there exists $\delta>0$ such that for all $t\in ]0,\delta[$
\begin{eqnarray*}
\alpha:=\sup_{\phi \in C} \left |\frac{\phi(g(\hat{x})+ t d_G g(\hat{x})(u)) - \phi(g(\hat{x}))- t d_G \phi(g(\hat{x}))(d_G g(\hat{x})(u))}{t} \right |<\varepsilon.
\end{eqnarray*}
Set
\begin{eqnarray*}
\beta &:=& \sup_{\phi \in C} \left |\frac{\phi\circ g(\hat{x}+tu) - \phi\circ g(\hat{x})- t d_G\phi(g(\hat{x}))(d_G g(\hat{x})(u))}{t} \right |.
\end{eqnarray*}
Using the triangular inequality, the continuity of $g$ at $\hat{x}$ and the fact that  $C$ is $r$-equi-Lipschitz at $g(\hat{x})$, we get that for $t>0$ small enough
\begin{eqnarray*}
\beta &<& \sup_{\phi \in C} \left |\frac{\phi (g(\hat{x}+tu)) - \phi(g(\hat{x})+ t d_G g(\hat{x})(u))}{t} \right | +\alpha \\
                                                                                                                                                               &<&  r \left \|\frac{g(\hat{x}+tu)-g(\hat{x})- t d_G g(\hat{x})(u)}{t} \right\|_Y +\varepsilon.
\end{eqnarray*}
Since $g$ is Gateaux differentiable at $\hat{x}$, we get using the above inequality and the expression of $\beta$ that the family $\lbrace \phi \circ g: \phi \in C\rbrace$ is equi-Gateaux differentiable at $\hat{x}$ and that $d_G (\phi\circ g)(\hat{x})=d_G\phi(g(\hat{x})) \circ d_G g(\hat{x})$ for all $\phi \in C$. On the other hand, since $C$ is $r$-equi-Lipschitz at $g(\hat{x})$, then we see easily that $\overline{\textnormal{conv}}^{w^*}\lbrace d_G \phi(g(\hat{x})): \phi\in C\rbrace \subset B_{Y^*}(0,r)$.  It follows that $\overline{\textnormal{conv}}^{w^*}\lbrace d_G \phi(g(\hat{x})): \phi\in C\rbrace$,  is a $w^*$-compact  subset of $Y^*$. This ends the proof.
\end{proof}

\begin{proposition} \label{FarkasM}  Let $E$  be an  l.c.t.v space,  $\hat{x}\in \Omega$ be an open subset of $E$ and $Y$ be a normed space. Let $g: \Omega \to Y$ and  $f: \Omega \to \R$ be two  mappings Gateaux differentiable at $\hat{x}$ and $g$ continuous at $\hat{x}$. Let  $A:=[C]^{\times}$ be an admissible set at $g(\hat{x})$ determined by a family $C$ of functions from $Y$ into $\R$.  Assume that  $\hat{x} \in \Omega$  is an optimal solution of the problem $(\mathcal{P})$:
\begin{equation*}
(\mathcal{P}_g)
\left \{
\begin{array}
[c]{l}
\max f\\
x\in \Omega \\
g(x) \in A.
\end{array}
\right. 
\end{equation*}

Then,  there exists $(\lambda^*, \beta^*)\in  \R^+\times \R^+ $ and $y^*\in \mathcal{T}_{C}(g(\hat{x}))$  such that 

$(a)$ $(\lambda^*, \beta^*)\neq (0,0)$, $y^*\neq 0$.

$(b)$ $\lambda^*  d_G f(\hat{x}) + \beta^* y^*\circ d_G g(\hat{x})=0$.

\end{proposition}
\begin{proof} Using the fact that  $A=[C]^{\times}$ is an admissible set at $g(\hat{x})$ together with Lemma \ref{Gateaux} we get that:

$(i)$ the set $\overline{\textnormal{conv}}^{w^*}\lbrace d_G\phi(g(\hat{x})): \phi\in C\rbrace$,  is a $w^*$-compact  subset of $Y^*$. 

$(ii)$  the family $\lbrace \phi\circ g: \phi \in C \rbrace$ is equi-Gateaux differentiable at $\hat{x}$  and $d_G(\phi\circ g)(\hat{x})=d_G\phi(g(\hat{x}))\circ d_G g(\hat{x})$ for all $\phi \in C$.

$(iii)$ the set $\lbrace \phi\circ g: \phi \in C \rbrace$ is equi-continuous at $\hat{x}$, since  $C$ is $r$-equi-Lipschitz  at $g(\hat{x})$ and $g$ is continuous at $\hat{x}$.

$(iv)$ $0\not \in \overline{\textnormal{conv}}^{w^*}\lbrace d_G \phi(g(\hat{x})): \phi\in C\rbrace$.

\noindent From $(iv)$, since $\mathcal{T}_{C}(g(\hat{x}))\subset \overline{\textnormal{conv}}^{w^*}\lbrace d_G \phi(g(\hat{x})): \phi\in C\rbrace$, we have that $0\not \in \mathcal{T}_{C}(g(\hat{x}))$. With the family $C_g:=\lbrace \phi \circ g: \phi \in C\rbrace$ and $F=[C_g]^\times$, we see that $g(x)\in A=[C]^{\times}$ if and only if $x\in F=[C_g]^\times$. From $(i)$, $(ii)$ and $(iii)$ we see easily that $F$ is weak-admissible at $\hat{x}$. Clearly, $\hat{x}$ is an optimal solution of the problem 
\begin{equation*}
(\mathcal{P})
\left \{
\begin{array}
[c]{l}
\max f\\
x\in \Omega \\
x \in F.
\end{array}
\right. 
\end{equation*}

Now, if $\hat{x}\in \textnormal{int}(F)$, then $d_G f(\hat{x})=0$ and so the theorem works with $\lambda^*=1$, $\beta^*=0$ and any $y^*\in\mathcal{T}_{C}(g(\hat{x}))$. If  $\hat{x}\not \in \textnormal{int}(F)$, by Theorem \ref{FarkasMM} applied to $F$, there exists $(\lambda^*,  \beta^*)\in  \R^+\times \R^+$  and  $x^*\in \mathcal{T}_{F,C_g}(\hat{x})\subset E^*$ such that

$(a)$ $(\lambda^*, \beta^*)\neq (0,0)$.

$(b)$ $\lambda^*  d_G f(\hat{x}) + \beta^* x^*=0$,

\noindent where, 
\begin{eqnarray*}
\mathcal{T}_{C_g}(\hat{x})&:=&\cap_{n\geq 1}\overline{\textnormal{conv}}^{w^*}\lbrace d_G(\phi\circ g)(\hat{x}): \phi(g(\hat{x})) \in [0,\frac{1}{n}], \phi \in C \rbrace\\
&=& \cap_{n\geq 1}\overline{\textnormal{conv}}^{w^*}\lbrace d_G\phi(g(\hat{x}))\circ d_G g(\hat{x}): \phi(g(\hat{x})) \in [0,\frac{1}{n}], \phi \in C \rbrace\\
&=& \lbrace y^* \circ d_G g(\hat{x}): y^*\in \mathcal{T}_{C}(g(\hat{x}))\subset Y^*\rbrace.
\end{eqnarray*}
 Hence, from $(b)$ we have $\lambda^*  d_G f(\hat{x}) + \beta^* y^* \circ d_G g(\hat{x})=0$ for some $y^*\in \mathcal{T}_{C}(g(\hat{x}))$, with $y^*\neq 0$ since  $0\not \in \mathcal{T}_{C}(g(\hat{x}))\neq \emptyset$ by Example \ref{nonzero}. This completes the proof. 

\end{proof}

Our results can also be applied when the constraints are in integral form. An example is given in the following corollary.  Let $K$ be a Hausdorff compact space and $X$ be a Banach space. The space $C(K, X)$ denotes the Banach space of $X$-valued bounded continuous functions on $K$ equipped with the sup-norm.  If $X=\R$, we simply note $C(K)$.

\begin{corollary} Let  $E$ be an l.c.t.v space  and $K$ be an Hausdorff topological space.  Let  $\Omega$ be an open subset of $E$, $S\subset \mathcal{M}^1(K)$ (the set of all Borel probability measures),  $g: \Omega \to C(K)$, $f: \Omega \to \R$ be two  mappings Gateaux differentiable at $\hat{x} \in \Omega$ and $g$ continuous at $\hat{x}$.  Assume that $\hat{x}$ is an optimal solution of the problem:
\begin{equation*}
(\mathcal{P})
\left \{
\begin{array}
[c]{l}
\max f\\
x \in \Omega \\
\int_K g(x) d\mu \geq 0: \forall \mu\in S
\end{array}
\right. 
\end{equation*}
  
Then,  there exists $(\lambda_0, \beta_0)\in  \R^+ \times \R^+ \setminus \lbrace (0,0)\rbrace$ and a Borel probability measure $\mu_0\in \overline{\textnormal{conv}}^{w^*}(S)$ on $K$ such that for all $x\in E$, $$\lambda_0  \langle d_G f(\hat{x}), x \rangle +  \beta_0 \int_{K} \langle d_G g(\hat{x}), x\rangle d\mu_0 =0$$
and $\int_K g(\hat{x}) d\mu_0=0$.
\end{corollary}
\begin{proof} We apply Proposition \ref{FarkasM} with $Y=C(K)$ and the family $A=[S]^\times$ which is an admissible set, noticing that $\mathcal{T}_{S}(g(\hat{x}))\subset \lbrace \mu \in \overline{\textnormal{conv}}^{w^*}(S): \int_K g(\hat{x}) d\mu=0\rbrace$.
\end{proof}
\begin{Exemp} The above corollary applies with $S:=\lbrace \delta_s : s\in K\rbrace\subset \mathcal{M}^1(K)$,  where for every $s\in K$, we denote $\delta_s : C(K)\to \R$, the evaluation map at $s$, defined by $\delta_s(z)=z(s)$ for all $z\in C(K)$,  $g: \Omega\subset  C(K, X)\to C(K)$ of the form $g(z):=h(\cdot, z(\cdot)): s \mapsto h(s, z(s))$, where $h: K\times X \to \R$ is a continuous function and equi-Gateaux differentiable on the second variable, that is, the family $(h(s,\cdot))_{s\in K}$ is equi-Gateaux differentiable at every point $x\in X$. In this case, $g$ is Gateaux differentiable at every $\hat{z}\in  C(K, X)$ and we have for all $z\in C(K, X)$: $\langle d_G g(\hat{z}), z \rangle: s\mapsto \langle D_{G,2}h(s,\hat{z}(s)), z(s) \rangle$, where $D_{G,2}h$ denotes the Gateaux-differential of $h$ with respect to the second variable.
\end{Exemp}

\subsection{Optimization with inequality and equality constraints.}
We give below our second main result which generalize the result of J. Jahn in \cite[Theorem 5.3 ]{Ja}. Let $E$ and $W$  be Banach spaces, $\Omega$ be an open subset of $E$. Let $Y$ be a normed  space and $A\subset Y$.  Let  $g: \Omega \to Y$, $f: \Omega \to \R$ and $h: \Omega \to W$ be  mappings. Consider the following problem:
\begin{equation*}
(\widetilde{\mathcal{P}})
\left \{
\begin{array}
[c]{l}
\max f\\
x\in \Omega \\
g(x) \in A\\
h(x)=0
\end{array}
\right. 
\end{equation*}
 Using the implicit function theorem, we will reduce the problem $(\widetilde{\mathcal{P}})$ to the problem $(\mathcal{P})$ without equality constraints,  then we apply Proposition \ref{FarkasM}. 

\begin{theorem} \label{Farkas-bis}   Let $\hat{x} \in \Omega$ and suppose that $A=Y$ or $A=[C]^{\times}\subset Y$ is an admissible set at $g(\hat{x})$ determined by a family $C$ of functions on $Y$. Assume that

$(\alpha)$  $\hat{x}$ is a solution of the problem $(\widetilde{\mathcal{P}})$.

$(\beta)$ $f$ and $g$ are  Gateaux differentiable at $\hat{x}$ and Lipschitz in a neighborhood of  $\hat{x}$.

$(\gamma)$ $h$ is  Fr\'echet differentiable in a neighborhood of  $\hat{x}$, $d_F h(\cdot)$ $($the Fr\'echet-differential of $h$$)$ is continuous at $\hat{x}$ and $\textnormal{Im} d_F h(\hat{x})$ is closed.

$(\sigma)$ $\textnormal{Ker}(d_F h(\hat{x}))$ is a complemented subspace of $E$, that is there exists a closed subspace $E_1$ of $E$ such that $E=\textnormal{Ker}(d_F h(\hat{x}))\oplus E_1$.
\vskip5mm
Then, there exists $\lambda^*_0\in  \R^+$, $z^*_0\in Y^*$  and $w^*_0 \in W^*$ such that  $(\lambda^*_0, z^*_0,w^*_0)\neq (0,0,0)$ and 
$$\lambda^*_0   d_G f(\hat{x})+  z^*_0 \circ d_G g(\hat{x})+w^*_0 \circ d_F h(\hat{x})=0,$$
where, $(\lambda^*_0,z_0^*,w^*_0)$ can be chosen as follows:

$\bullet$  If $d_F h(\hat{x})$ is not onto: $(\lambda^*_0, z^*_0,w^*_0)=(0,0, w^*_0)$, with $w^*_0\neq 0$.

$\bullet$  If $d_F h(\hat{x})$ is onto and $A=Y$: $(\lambda^*_0, z^*_0,w^*_0)=(1,0, 0)$. 

$\bullet$ If $d_F h(\hat{x})$ is onto and $A=[C]^\times$: $(\lambda^*_0, z^*_0)\neq (0,0)$, with $z^*_0\in \R^+ \mathcal{T}_{C}(g(\hat{x}))$, where $0\not \in  \mathcal{T}_{C}(g(\hat{x}))$.
\end{theorem}

\begin{proof} We have two cases:

{\bf Case 1. $\textnormal{Im} d_F h(\hat{x})\neq W$.} Since $\textnormal{Im} d_F h(\hat{x})$ is closed, then by the Hahn-Banach theorem, there exists $w^*\in W^*\setminus \lbrace 0\rbrace$ such that $w^*\circ d_F h(\hat{x})=0$. Thus, 
the theorem works with $\lambda_0=0$,  $y^*_0=0$ and $w^*\neq 0$.

{\bf Case 2. $\textnormal{Im} d_F h(\hat{x})= W$.} In this case, $d_F h(\hat{x})_{|E_1}: E_1 \to W$ is an isomorphism.  Let $(\hat{a},\hat{b})\in \textnormal{Ker}(d_F h(\hat{x}))\times E_1$ such that $\hat{x}=\hat{a}+\hat{b}$. Since, $h(\hat{x})=0$, then  by the implicit function theorem, there exists a neighborhood $U$ of $\hat{a}$ in $\textnormal{Ker}(d_F h(\hat{x}))$, a neighborhood $V$  of $\hat{b}$ in $E_1$ such that $U+V\subset \Omega$ and a unique continuous function  $\varphi : U\to V$  such that

$(i)$ $\varphi(\hat{a})=\hat{b}$.

$(ii)$ $\forall x\in U$, $h(x+\varphi(x))=0$.

$(iii)$ $\varphi$ is  Fr\'echet  differentiable at $\hat{a}$, and $d_F \varphi(\hat{a})=0$.
\vskip5mm
Let us define $\widehat{f} : U \to \R$, $l: U\to Y$ by $\widehat{f}(x)=f(x+\varphi(x))$ and $l(x)=g(x+\varphi(x))$ for all $x\in U\subset \textnormal{Ker}(d_F h(\hat{x}))$. Notice that $l(\hat{a})=g(\hat{x})$ and $\widehat{f}(\hat{a})=f(\hat{x})$. By assumption and part $(ii)$ above, we have that $\hat{a}$ is a solution of 
\begin{equation*}
(\mathcal{P})
\left \{
\begin{array}
[c]{l}
\max \widehat{f}\\
x\in U \\
l(x) \in A.
\end{array}
\right. 
\end{equation*}
If $A=Y$, we see that $d_G \widehat{f}(\hat{a})=0$. If $A=[C]^{\times}$ is an admissible set at $g(\hat{x})=l(\hat{a})$, we get using Proposition \ref{FarkasM} with $\widehat{f}$ and $l$, that there exists $(\lambda^*_0, \beta^*_0)\in  \R^+\times \R^+ $ and $y^*_0\in \mathcal{T}_{C}(l(\hat{a}))$  such that  $(\lambda^*, \beta^*)\neq (0,0)$, $y^*_0\neq 0$ and
\begin{eqnarray}\label{formula10}
\lambda^*_0  d_G \widehat{f}(\hat{a})+\beta^*_0 y_0^*\circ d_G l(\hat{a})=0, 
\end{eqnarray}

Now, we observe from the expressions of  $\widehat{f}$ and $l$, using $(iii)$ and the fact that $f$ and $g$ are Lipschitz in a neighborhood of  $\hat{x}$, that  
$$d_G \widehat{f}(\hat{a})= d_G f(\hat{x})\circ (I_{\textnormal{Ker}(d_F h(\hat{x}))}+d_F \varphi(\hat{a}))=d_G f(\hat{x})\circ I_{\textnormal{Ker}(d_F h(\hat{x}))},$$ 
$$d_G l(\hat{a})= d_G g(\hat{x})\circ (I_{\textnormal{Ker}(d_F h(\hat{x}))}+d_F \varphi(\hat{a}))=d_G g(\hat{x})\circ I_{\textnormal{Ker}(d_F h(\hat{x}))}.$$ 
where $I_{\textnormal{Ker}(d_F h(\hat{x}))}$ denotes  the identity map from $ \textnormal{Ker}(d_F h(\hat{x}))$ into $E$. Thus, from $(\ref{formula10})$ we have 
\begin{eqnarray}\label{formula21}
\lambda^*_0 d_G f(\hat{x})\circ I_{\textnormal{Ker}(d_F h(\hat{x}))}+  \beta^*_0 y_0^* \circ d_G g(\hat{x})\circ I_{\textnormal{Ker}(d_F h(\hat{x}))}= 0
\end{eqnarray}
We set
\begin{eqnarray*}
w^*_0:=(-\lambda^*_0  d_G f(\hat{x})_{|E_1}- \beta^*_0 y_0^* \circ d_G g(\hat{x})_{|E_1})\circ (d_F h(\hat{x})_{|E_1})^{-1}\in W^*
\end{eqnarray*}
Then, we have
\begin{eqnarray}\label{formula22}
\lambda^*_0 d_G f(\hat{x})_{|E_1}+ \beta^*_0 y_0^* \circ d_G g(\hat{x})_{|E_1}  +w^*_0\circ d_F h(\hat{x})_{|E_1} =0.
\end{eqnarray}
Since $ d_F h(\hat{x})\circ I_{\textnormal{Ker}(d_F h(\hat{x}))}=0$, by using the formulas (\ref{formula21}) and (\ref{formula22}), we obtain
\begin{eqnarray*}
\lambda^*_0 d_G f(\hat{x})+ \beta^*_0 y_0^* \circ d_G g(\hat{x}) +w^*_0\circ d_F h(\hat{x})=0,
\end{eqnarray*}
where,  $(\lambda^*_0, \beta^*_0y^*_0)\neq (0,0)$ since $(\lambda^*_0, \beta^*_0)\neq (0,0)$ and $y^*_0\neq 0$ (Notice that $0\not \in \mathcal{T}_{C}(g(\hat{x}))$, since $[C]^\times$ is admissible at $g(\hat{x})$).
\end{proof}

 We give the following corollary (an immediate consequence of Theorem \ref{Farkas-bis}) which  is also an extention of the result in \cite[Theorem 5.3 ]{Ja}. The following result extends results known for closed convex cones under the Fréchet differentiability hypothesis to the more general case of closed convex sets whose recession cones have nonempty interiors and under the hypothesis that $f$ and $g$ are Gateaux differentiable at the optimal solution.

\begin{corollary} \label{cor-Farkas-bis} Under the hypothesis of Theorem \ref{Farkas-bis}, assume that $A$ is a closed convex set such that $\textnormal{int}(\mathcal{R}_A)\neq \emptyset$ $($in particular if $A$ is a closed convex cone with a nonempty interior$)$. Then, there exists $\lambda^*_0\in  \R^+$, $z^*_0\in (\mathcal{R}_A)^*$  and $w^*_0 \in W^*$ such that  

$(i)$ $(\lambda^*_0, z^*_0,w^*_0)\neq (0,0,0)$,

$(ii)$ $\lambda^*_0   d_G f(\hat{x})+  z^*_0 \circ d_G g(\hat{x})+w^*_0 \circ d_F h(\hat{x})=0,$

$(iii)$ $z^*_0(g(\hat{x}))=\inf_{y\in A} z^*_0(y)$ $($$=0$, if $A$ is a closed convex cone, in this case, $\mathcal{R}_A=A$$)$, that is, $g(\hat{x})$  minimises  $z^*_0$ on $A$. 

\noindent The multipliers $(\lambda^*_0,z_0^*,w^*_0)$ can be chosen as follows:

$\bullet$  If $d_F h(\hat{x})$ is not onto: $(\lambda^*_0, z^*_0,w^*_0)=(0,0, w^*_0)$, with $w^*_0\neq 0$.

$\bullet$  If $d_F h(\hat{x})$ is onto and $A=Y$: $(\lambda^*_0, z^*_0,w^*_0)=(1,0, 0)$. 

$\bullet$ If $d_F h(\hat{x})$ is onto and $A\neq Y$: $(\lambda^*_0, z^*_0)\neq (0,0)$.
\end{corollary}
\begin{proof} We apply directly  Theorem \ref{Farkas-bis} using the following fact: in the case where $A\neq Y$ and $\textnormal{int}(\mathcal{R}_A)\neq \emptyset$, using Example \ref{nonzero}, we have that $A$ is admissible at each of its points and is determined by $C=\lbrace z^* - \inf_{x\in A} z^*(x): z^*\in S_{Y^*}\cap (-\textnormal{bar}(A))\rbrace$ and $\mathcal{T}_{C}(\hat{x}) \subset \lbrace z^*\in (\mathcal{R}_A)^*\setminus \lbrace 0\rbrace: z^*(\hat{x})=\inf_{x\in A} z^*(x)\rbrace$.
\end{proof}

\section*{Declaration} 

- The authors declare that there is no conflict of interest.

- Data sharing not applicable to this article as no datasets were generated or analysed during the current study.

\section*{Acknowledgement}
This research has been conducted within the FP2M federation (CNRS FR 2036) and  SAMM Laboratory of the University Paris Panthéon-Sorbonne.

\bibliographystyle{amsplain}

\end{document}